\documentclass[9pt]{amsart}
\usepackage{amsmath,amsfonts,amsthm,amssymb}
\usepackage{color}
\usepackage{verbatim}
\usepackage{eepic,epic}
\usepackage{epsfig,subfigure}
\usepackage{epstopdf}
\usepackage[]{graphics}
\usepackage{url}

\newtheorem{theorem}{Theorem}[section]
\newtheorem{lemma}{Lemma}[section]

\newtheorem{remark}{Remark}[section]

\numberwithin{equation}{section}

\newcommand{\al}{\alpha}
\newcommand{\la}{\lambda}
\newcommand{\g}{\gamma}
\newcommand{\DD}{\partial_t^\alpha}
\newcommand{\II}{(\Omega)}
\newcommand{\fy}{\varphi}
\def\dH#1{\dot H^{#1}(\Omega)}
\def\bPtau{\bar\partial_\tau}

\begin{document}

\title[Rayleigh-Stokes Problem]
{An Analysis of the Rayleigh-Stokes problem for a Generalized Second-Grade Fluid}

\author{Emilia Bazhlekova}
\address{Institute of Mathematics and
Informatics, Bulgarian Academy of Sciences, Acad. G. Bonchev str., Bl. 8, Sofia 1113, Bulgaria}
\email{e.bazhlekova@math.bas.bg}

\author {Bangti Jin}
\address{Department of Computer Science, University College London, Gower Street, London WC1E 6BT, UK}
\email{bangti.jin@gmail.com}

\author {Raytcho Lazarov}
\address{Department of Mathematics, Texas A\&M
University, College Station, TX, 77843,
USA and Institute of Mathematics and
Informatics, Bulgarian Academy of Sciences, Acad. G. Bonchev str., Bl. 8, Sofia 1113, Bulgaria}
\email{lazarov@math.tamu.edu}
\author {Zhi Zhou}
\address{Department of Mathematics, Texas A\&M
University, College Station, TX, 77843,
USA}\email{zzhou@math.tamu.edu}
\date{started August 1, 2013, today is \today}

\begin{abstract}
We study the Rayleigh-Stokes problem for a generalized second-grade fluid which involves a Riemann-Liouville
fractional derivative in time, and present an analysis of the problem in the continuous, space semidiscrete
and fully discrete formulations. We establish the Sobolev regularity of the homogeneous problem for both
smooth and nonsmooth initial data $v$, including $v\in L^2(\Omega)$. A space semidiscrete
Galerkin scheme using continuous piecewise linear finite elements is developed,
and optimal with respect to initial data regularity error estimates for the finite element
approximations are derived. Further, two fully discrete schemes based on the backward Euler method and
second-order backward difference method and the related convolution quadrature are developed, and optimal error
estimates are derived for the fully discrete approximations for both smooth and nonsmooth initial data.
Numerical results for one- and two-dimensional examples with smooth and nonsmooth initial data are presented
to illustrate the efficiency of the method, and to verify the convergence theory.
\end{abstract}

\maketitle

\section {Introduction}

In this paper, we study the homogeneous Rayleigh-Stokes problem for a generalized second-grade
fluid with a fractional derivative model. Let $\Omega\subset\mathbb{R}^d$ ($d=1,2,3$)
be a convex polyhedral domain with its boundary being $\partial\Omega$, and $T>0$ be a fixed
time. Then the mathematical model is given by
\begin{equation}\label{eqn:rsp}
  \begin{aligned}
\partial_t u-(1+\gamma \DD)\Delta u&=f,\ \ \mbox{in }\Omega,\ \  0<t\le T;\\ 
u&=0,\ \  \mbox{on }\partial\Omega, \ 0<t\le T;\\ 
u(\cdot,0)&=v,\ \ \mbox{in }\Omega, 
  \end{aligned}
\end{equation}
where $\gamma>0$ is a fixed constant, $v$ is the initial data, $\partial_t=\partial/\partial t$, and $\DD$
is the Riemann-Liouville fractional derivative of order $\alpha\in(0,1)$ defined by \cite{KilbasSrivastavaTrujillo:2006,Podlubny:1999}:
\begin{equation*}
\DD f(t)=\frac{d}{d t}\int_0^t \omega_{1-\alpha}(t-s)f(s)\,ds,
\qquad \omega_\alpha(t)=\frac{t^{\alpha-1}}{\Gamma(\alpha)}.
\end{equation*}

The Rayleigh-Stokes problem \eqref{eqn:rsp} has received considerable attention in recent years. The
fractional derivative $\partial_t^\alpha$ in the model is used to capture the viscoelastic behavior
of the flow; see e.g. \cite{ShenTanZhaoMasuoka:2006,FetecauJamilFetecauVieru:2009} for derivation
details. The model \eqref{eqn:rsp} plays an important role in describing the behavior of some
non-Newtonian fluids.

In order to gain insights into the behavior of the solution of this model, there has been substantial interest
in deriving a closed form solution for special cases; see, e.g. \cite{ShenTanZhaoMasuoka:2006,ZhaoYang:2009,
FetecauJamilFetecauVieru:2009}. For example, Shen et al \cite{ShenTanZhaoMasuoka:2006} obtained the exact
solution of the problem using the Fourier sine transform and fractional Laplace transform. Zhao and Yang
\cite{ZhaoYang:2009} derived exact solutions using the eigenfunction expansion on a rectangular domain for the
case of homogeneous initial and boundary conditions. The solutions obtained in these studies are
formal in nature, and especially the regularity of the solution has not been studied. In Section \ref{sec:reg}
below, we fill this gap and establish the Sobolev regularity of the solution for both smooth and nonsmooth initial data.
We would like to mention that Girault and Saadouni \cite{GiraultSaadouni:2007} analyzed the
existence and uniqueness of a weak solution of a closely related time-dependent grade-two fluid model.

The exact solutions obtained in these studies involve infinite series and special functions, e.g., 
generalized Mittag-Leffler functions, and thus are inconvenient for numerical evaluation. Further, 
closed-form solutions are available only for a restricted class of problem settings. Hence, it is 
imperative to develop efficient and optimally accurate numerical algorithms for problem 
\eqref{eqn:rsp}. This was considered earlier in \cite{ChenLiuAnh:2008,ChenLiuAnh:2009,LinJiang:2011,
MohebbiAbbaszadehDehghan:2013,Wu:2009}. Chen et al \cite{ChenLiuAnh:2008} developed implicit and 
explicit schemes based on the finite difference method in space and the Gr\"{u}nwald-Letnikov
discretization of the time fractional derivative, and analyzed their stability and convergence rates 
using the Fourier method. Of the same flavor is the work \cite{ChenLiuAnh:2009},
where a scheme based on Fourier series expansion was considered. Wu \cite{Wu:2009} developed
an implicit numerical approximation scheme by transforming problem
\eqref{eqn:rsp} into an integral equation, and showed its stability and convergence by an energy argument.
Lin and Jiang \cite{LinJiang:2011} described a method based on the reproducing kernel Hilbert space.
Recently, Mohebbi et al \cite{MohebbiAbbaszadehDehghan:2013} compared a compact finite difference
method with the radial basis function method. In all these studies, however, the error estimates were
obtained under the assumption that the solution to \eqref{eqn:rsp} is sufficiently smooth and
the domain $\Omega$ is a rectangle. Hence the interesting cases of
nonsmooth data (the initial data or the right hand side) and general domains are not covered.

Theoretical studies on numerical methods for
differential equations involving fractional derivatives have received considerable
attention in the last decade.  McLean and Mustapha \cite{McLeanMustapha:2009,MustaphaMcLean:2011}
analyzed piecewise constant and piecewise linear discontinuous Galerkin method in time, and
derived error estimates for smooth initial data; see also \cite{MustaphaMcLean:2013} for related
superconvergence results. In \cite{JinLazarovZhou:2013,JinLazarovPasciakZhou:2013},
a space semidiscrete Galerkin finite element method (FEM) and lumped mass method for
problem ${^C \kern -.2em\partial_t^\alpha} u + A u = 0$ with $u(0)=v$ (with $A$ being an elliptic operator,
and ${^C\kern -.2em\partial_t^\alpha}$ being the Caputo derivative) has been analyzed. Almost optimal
error estimates were established for initial data $v\in \dH q$, $-1\leq q\leq 2$,
(see Section \ref{sec:reg} below for the definition) by exploiting the properties
of the two-parameter Mittag-Leffler function. Note that this includes weak (nonsmooth),
$ v \in L^2(\Omega)$, and very weak data, $v\in \dH {-1}$.
In \cite[Section 4]{McLeanThomee:2010a}, McLean and Thom\'{e}e studied the following equation
$\partial_t u + {\partial_t^{-\alpha}} Au = f$ (with ${\partial_t^{-\alpha}}$ being Riemann-Liouville
integral and derivative operator for $\alpha\in(0,1)$ and $\alpha\in(-1,0)$, respectively), and derived
$L^2\II$-error estimates for the space semidiscrete scheme for both $v\in L^2(\Omega)$ and
$v\in \dH 2$ (and suitably smooth $f$) and some fully discrete schemes based on Laplace transform
were discussed. The corresponding $L^\infty(\Omega)$ estimates for data $v\in L^\infty(\Omega)$
and $Av\in L^\infty(\Omega)$ were derived in \cite{McLeanThomee:2010b}. Lubich et al
\cite{LubichSloanThomee:1996} developed two fully discrete schemes for the problem
$\partial_tu + \partial_t^{-\alpha} Au = f$ with $u(0)=v$ and $0<\alpha<1$ based
on the convolution quadrature of the fractional derivative term, and derived optimal
error estimates for nonsmooth initial data and right hand side. Cuesta et al
\cite{CuestaLubichPlencia:2006} considered the semi-linear counterpart of the model with convolution
quadrature, which covers also the fractional diffusion case, i.e., $-1<\alpha<0$, and provided
a unified framework for the error analysis with optimal error estimates in an abstract Banach space setting.

In this paper we develop a Galerkin FEM for problem \eqref{eqn:rsp} and derive optimal with respect to 
data regularity error estimates for both smooth and nonsmooth initial data. The approximation is based on
the finite element space $X_h$ of continuous piecewise linear functions over a family of shape regular 
quasi-uniform partitions $\{\mathcal{T}_h\}_{0<h<1}$ of the domain $\Omega$ into $d$-simplexes, where $h$ 
is the maximum diameter. The semidiscrete Galerkin FEM for problem \eqref{eqn:rsp} is: find $ u_h (t)\in X_h$ such that
\begin{equation}\label{eqn:fem}
  \begin{split}
   {(\partial_t u_{h},\chi)}+ \g\partial_t^\alpha a(u_h,\chi) + a(u_h,\chi)&= {(f, \chi)},
                      \quad \forall \chi\in X_h,\ T \ge t >0,
   \quad u_h(0)=v_h,
  \end{split}
\end{equation}
where $a(u,w)=(\nabla u, \nabla w) ~~ \text{for}\ u, \, w\in H_0^1(\Omega)$, and $v_h \in X_h$ is an
approximation of the initial data $v$. Our default choices are the $L^2(\Omega)$ projection $v_h=P_hv$, assuming
$v\in L^2(\Omega)$, and the Ritz projection $v_h=R_hv$, assuming $v\in
\dH 2$. Further, we develop two fully discrete schemes based on the backward Euler method and the
second-order backward difference method and the related convolution quadrature for the fractional derivative term,
which achieves respectively first and second-order accuracy in time. Error estimates optimal with respect to
data regularity are provided for both semidiscrete and fully discrete schemes.

Our main contributions are as follows. First, in Theorem \ref{thm:reg}, using an operator
approach from \cite{Pruss:1993}, we develop the theoretical foundations for our study by
establishing the smoothing property and decay behavior of the solution to problem \eqref{eqn:rsp}.
Second, for both smooth initial data
$v\in \dH 2$ and nonsmooth initial data $v\in L^2(\Omega)$, we derive error estimates for the space
semidiscrete scheme, cf. Theorems \ref{thm:nonsmooth-initial-op} and  \ref{thm:smooth-initial-op}:
\begin{equation*}
 \|u(t)-u_h(t)\|_{L^2(\Omega)} + h\|\nabla (u(t)-u_h(t))\|_{L^2(\Omega)} \leq ch^2 t^{(q/2-1)(1-\alpha)}\|v\|_{\dH q},\quad q =0,2.
\end{equation*}
The estimate for $v\in L^2(\Omega)$ deteriorates as $t$ approaches $0$. The error estimates are
derived following an approach due to Fujita and Suzuki \cite{FujitaSuzuki:1991}. Next, in Theorems
\ref{thm:error_fully_smooth} and \ref{thm:error_fully_smooth2} we establish optimal $L^2(\Omega)$
error estimates for the two fully discrete schemes. The proof is inspired by the fundamental work
of Cuesta et al \cite{CuestaLubichPlencia:2006}, which relies on known error estimates for convolution
quadrature and bounds on the convolution kernel.  We show for example,
that the discrete solution $U_h^n$ by the
backward Euler method (on a uniform grid in time with a time step size $\tau$)
satisfies the following {\it a priori} error bound
\begin{equation*}
  \|U_h^n - u(t_n)\|_{L^2(\Omega)}
  \leq c(\tau t_n^{-1 + (1-\alpha)q/2}+h^2 t_n^{(q/2-1)(1-\alpha)})\|v\|_{\dH q},\quad q =0,2.
\end{equation*}
A similar estimate holds for the second-order backward difference method.

The rest of the paper is organized as follows. In Section \ref{sec:reg} we establish the Sobolev
regularity of the solution. In Section \ref{sec:semidiscrete},
we analyze the space semidiscrete scheme, and derive optimal error estimates for both smooth and
nonsmooth initial data. Then in Section \ref{sec:fullydiscrete}, we develop two fully discrete
schemes based on convolution quadrature approximation of the fractional derivative. Optimal error
estimates are provided for both schemes. Finally in Section \ref{sec:numeric}, numerical results for
one- and two-dimensional examples are provided to illustrate the convergence theory. Throughout,
the notation $c$ denotes a constant which may differ at different occurrences, but it is always
independent of the solution $u$, mesh size $h$ and time step-size $\tau$.

\section{Regularity of the solution}\label{sec:reg}

In this section, we establish the Sobolev regularity of the solution to \eqref{eqn:rsp} in the homogeneous case $f\equiv 0$. We first
recall preliminaries on the elliptic operator and function spaces. Then we derive the proper
solution representation, show the existence of a weak solution, and establish the Sobolev regularity of
the solution to the homogeneous problem. The main tool is the operator theoretic approach
developed in \cite{Pruss:1993}. Further, we give an alternative solution representation
via eigenfunction expansion, and derive qualitative properties of the time-dependent components.

\subsection{Preliminaries}
First we introduce some notation. For
$q\ge-1$, we denote by $\dH {q} \subset H^{-1}(\Omega)$ the Hilbert space induced by the norm
\begin{equation*}
   \|v\|_{\dH q}^2=\sum_{j=1}^\infty\lambda_j^q ( v,\fy_j)^2,
\end{equation*}
with $(\cdot,\cdot)$ denoting the inner product in $L^2(\Omega)$ and $\{\lambda_j\}_{j=1}^\infty$ and
$\{\fy_j\}_{j=1}^\infty$ being respectively the Dirichlet eigenvalues and eigenfunctions of $-\Delta$
on the domain $\Omega$. As usual, we identify a function $f$
in $L^2(\Omega)$ with the functional $F$ in $H^{-1}(\Omega)\equiv (H_0^1\II)'$ defined by $\langle F,\phi\rangle =
(f,\phi)$, for all $\phi\in H^1_0(\Omega)$. Then sets $\{\fy_j\}_{j=1}^\infty$ and $\{\la_j^{1/2}
\fy_j\}_{j=1}^\infty$ form orthonormal basis in $L^2(\Omega)$ and $H^{-1}(\Omega)$, respectively.
Thus $\| v \|_{\dH {-1}}=\| v \|_{H^{-1}\II}$,
$\|v\|_{\dH 0}=\|v\|_{L^2\II}=(v,v)^{1/2}$ is the norm
in $L^2(\Omega)$, $\|v\|_{\dH 1}$ is the norm in $H_0^1(\Omega)$ and $\|v\|_{\dH 2}=\|\Delta v\|_{L^2\II}$ is
equivalent to the norm in $H^2(\Omega)$ when $v=0$ on $\partial\Omega$ \cite{Thomee:2006}.
Note that $\dot H^s(\Omega)$, $s\ge -1$ form a
Hilbert scale of interpolation spaces. Thus we denote $\|\cdot\|_{H_0^s(\Omega)}$ to
be the norm on the interpolation scale between $H^1_0(\Omega)$ and $L^2(\Omega)$  for  $s$
is in the interval $[0,1]$ and $\|\cdot\|_{H^{s}(\Omega)}$ to be the norm on the interpolation scale between
$L^2(\Omega)$ and $H^{-1}(\Omega)$ when $s$ is in $[-1,0]$.  Then, the $\dot H^s\II$ and $H_0^s\II$ norms 
are equivalent for any $s\in [0,1]$ by interpolation, and likewise the $\dot H^s\II$ and $H^{s}\II$ norms
are equivalent for any $s\in[-1,0]$. 

For $\delta>0$ and $\theta\in(0,\pi)$ we introduce the contour $\Gamma_{\delta,\theta}$  defined by
\begin{equation*}
   \Gamma_{\delta,\theta}=\left\{re^{-\mathrm{i}\theta}:\ r\geq \delta\right\}\cup \left\{\delta e^{i\psi}:
    \ |\psi|\le\theta\right\}\cup \left\{re^{\mathrm{i}\theta}:\ r\geq \delta\right\},
\end{equation*}
where the circular arc is oriented counterclockwise, and the two rays are oriented with an increasing
imaginary part. Further, we denote by $\Sigma_\theta$ the sector
\begin{equation*}
  \Sigma_\theta=\{z\in\mathbb{C};\ z\neq 0, |\arg z|<\theta\}.
\end{equation*}

We recast problem \eqref{eqn:rsp} with $f\equiv 0$ into a Volterra integral equation by integrating
both sides of the governing equation in \eqref{eqn:rsp}
\begin{equation}\label{eqn:V}
   u(x,t)=v(x)-\int_0^t k(t-s) A u(x,s)\,ds,
\end{equation}
where the kernel $k(t)$ is given by
\begin{equation*}
  k(t)=1+\g\omega_{1-\alpha}(t)
\end{equation*}
and the operator $A$ is defined by $A=-\Delta$ with a domain $D(A)= H_0^1(\Omega) \cap H^2(\Omega)$.
The $H^2(\Omega)$ regularity of the elliptic problem is essential for our discussion, and it follows
from the convexity assumption on the domain $\Omega$.
It is well known that the operator $-A$ generates a bounded analytic semigroup of angle $\pi/2$, i.e.,
for any $\theta \in (0,\pi/2)$
\begin{equation}\label{eqn:resolvent}
  \|(z+A)^{-1}\|\le M/|z|,\ \ \forall z\in\Sigma_{\pi-\theta}.
\end{equation}
Meanwhile, applying the Laplace transform to \eqref{eqn:V} yields
\begin{equation*}
   \widehat u (z) + \widehat{k}(z) A\widehat u(z) = z^{-1}v,
\end{equation*}
i.e., $\widehat{u}(z) = H(z)v$, with the kernel $H(z)$ given by
\begin{equation}\label{eqn:HH}
    H(z)=\frac{g(z)}{z}(g(z)I+A)^{-1},\qquad
    g(z)=\frac{1}{\widehat{k}(z)}=\frac{z}{1+\gamma z^\alpha},
\end{equation}
where $\widehat{k}$ is the Laplace transform of the function $k(t)$. Hence, by means of the inverse
Laplace transform, we deduce that the solution operator $S(t)$ is given by
\begin{equation}\label{eqn:solop}
  S(t)=\frac{1}{2\pi\mathrm{i}}\int_{\Gamma_{\delta,\pi-\theta}} e^{z t} H(z)\, dz,
\end{equation}
where $\delta>0$, $\theta \in (0,\pi/2)$.

First we state one basic estimate about the kernel $g(z)=z/(1+\gamma z^\alpha)$.
\begin{lemma}\label{lem:est-g}
Fix $\theta\in(0,\pi)$, and let $g(z)$ be defined in \eqref{eqn:HH}. Then
\begin{equation}\label{eqn:controlg1}
 g(z) \in \Sigma_{\pi-\theta}\quad \mbox{and}\quad  |g(z)|\le c\min(|z|,|z|^{1-\alpha}),\quad \forall z\in \Sigma_{\pi-\theta}.
\end{equation}
\end{lemma}
\begin{proof}
Let $z \in \Sigma_{\pi-\theta}$, i.e. $z=r e^{\mathrm{i}\psi}$, $|\psi|<\pi-\theta$, $r>0$. Then by
noting $\alpha\in(0,1)$,
\begin{equation}\label{eqn:ginsector}
  \begin{aligned}
   g(z) = \frac{re^{\mathrm{i}\psi}}{1+\gamma r^\alpha e^{\alpha\mathrm{i}\psi}}
     = \frac{r e^{\mathrm{i}\psi}+\gamma r^{\al+1}e^{\mathrm{i}(1-\al)\psi}}{(1+\gamma
     r^{\al} \cos(\al\psi))^2+(\gamma r^\al\sin(\al\psi))^2}\in\Sigma_{\pi-\theta}.
  \end{aligned}
\end{equation}
To prove (\ref{eqn:controlg1}) we note that
\begin{equation}\label{eqn:gg}
  \begin{aligned}
    |1+\gamma z^\alpha|^2=1+2\gamma r^\alpha\cos(\alpha \psi)+\gamma^2 r^{2\alpha}
    >1+2\gamma r^\alpha\cos(\alpha \pi)+\gamma^2 r^{2\alpha}.
  \end{aligned}
\end{equation}
Let $b=\cos(\alpha \pi)$. Since the function $f(x)=1+2bx+x^2$ attains its minimum at $x=-b$,
with a minimum value $f_{min}=f(-b)=1-b^2$, it follows from \eqref{eqn:gg} that
\begin{equation*}
   |1+\gamma z^\alpha|^2>1-\cos^2(\alpha \pi)=\sin^2(\alpha\pi).
\end{equation*}
Since $\sin(\alpha\pi)>0$, this leads to the first assertion $$|g(z)|=\left|\frac{z}{1+\g z^\alpha}
\right|<\frac{1}{\sin(\alpha\pi)}|z|.$$

From \eqref{eqn:gg} it follows that
\begin{equation}\label{eqn:o1}
  |1+\gamma z^\alpha|^2>(1+\gamma r^{\al} \cos(\al\pi))^2+(\gamma r^\al\sin(\al\pi))^2\geq \sin^2(\alpha\pi)\gamma^2r^{2\alpha},
\end{equation}
and consequently, we get
\begin{equation*}
        |g(z)|  = \bigg|\frac{z}{1+\gamma z^\alpha}\bigg|
		 \leq \frac{r}{\gamma r^\alpha\sin(\alpha\pi)}=\frac{1}{\gamma\sin(\alpha\pi)}|z|^{1-\alpha}.
   \end{equation*}
This completes the proof of the lemma.
\end{proof}

\subsection{A priori estimates of the solution}

Now we can state the regularity to problem \eqref{eqn:rsp} with $f\equiv 0$.
\begin{theorem}\label{thm:reg}
For any $v\in L^2(\Omega)$ and $f\equiv 0$ there exists a unique solution $u$ to problem \eqref{eqn:rsp} and
\begin{equation*}
u=S(t)v\in C([0,T];L^2(\Omega))\cap C((0,T];H^2(\Omega)\cap H_0^1(\Omega)).
\end{equation*}
Moreover, the following stability estimates hold for $t\in(0,T]$ and $\nu=0,1$:
\begin{eqnarray}
&&\|A^\nu S^{(m)}(t)v\|_{L^2(\Omega)}\le c t^{-m-\nu(1-\al)}\|v\|_{L^2(\Omega)},\ v\in L^2(\Omega),
m\ge 0,\label{eqn:new1}\\
&&\|A^\nu S^{(m)}(t)v\|_{L^2(\Omega)}\le c_T t^{-m+(1-\nu)(1-\al)}\|Av\|_{L^2(\Omega)},\ v\in D(A), \nu+m\ge 1,\label{eqn:new2}
\end{eqnarray}
where $c,c_T>0$ are constants depending on $d$, $\Omega$, $\al$, $\g$, $M$ and $m$, and
the constant $c_T$ also depends on $T$.
\end{theorem}
\begin{proof}

By Lemma \ref{lem:est-g} and \eqref{eqn:resolvent} we obtain
\begin{equation}\label{eqn:resolventg}
  \|(g(z)I+A)^{-1}\|\le M/|g(z)|,\ \ z\in\Sigma_{\pi-\theta},
\end{equation}
and we deduce from \eqref{eqn:HH} and \eqref{eqn:resolventg} that
\begin{equation}\label{Hcontrol}
  \|H(z)\|\le {M}/{|z|},\ \ z\in\Sigma_{\pi-\theta}.
\end{equation}
Then by \cite[Theorem 2.1 and Corollary 2.4]{Pruss:1993}, for any $v\in L^2(\Omega)$ there
exists a unique solution $u$ of \eqref{eqn:V} and it is given by
\begin{equation*}
   u(t)=S(t)v.
\end{equation*}
It remains to show the estimates.

Let $t>0$, $\theta \in (0,\pi/2)$, $\delta>0$. We choose $\delta=1/t$ and denote for short
\begin{equation}\label{eqn:Gamma}
\Gamma=\Gamma_{1/t,\pi-\theta}.
\end{equation}
First we derive \eqref{eqn:new1} for $\nu=0$ and $m\ge 0$. From \eqref{eqn:solop} and \eqref{Hcontrol} we deduce
\begin{equation*}
 \begin{aligned}
 \|S^{(m)}(t)\| &=\left\|
	\frac{1}{2\pi \mathrm{i}} \int_{\Gamma} z^m e^{zt} H(z)\, dz
		\right\|
    \le c\int_{\Gamma} |z|^m e^{\Re(z)t} \|H(z)\|\, |dz|\\
	&\le c\left( \int_{1/t}^\infty r^{m-1}e^{-rt\cos\theta}  \,dr
  + \int_{-\pi+\theta}^{\pi-\theta}e^{\cos\psi}t^{-m}\,d\psi\right)\leq ct^{-m}.
	\end{aligned}
\end{equation*}

Next we prove estimate \eqref{eqn:new1} for $\nu=1$ and $m\ge 0$. By applying the operator $A$ to both
sides of \eqref{eqn:solop} and differentiating
we arrive at
\begin{equation}\label{eqn:Asolop}
AS^{(m)}(t)=\frac{1}{2\pi\mathrm{i}}\int_\Gamma z^m e^{z t} AH(z) dz.
\end{equation}
Using the identity
\begin{equation*}
   AH(z)=\left(-H(z)+z^{-1}I\right)g(z),
\end{equation*}
it follows from \eqref{Hcontrol} and Lemma \ref{lem:est-g} that
\begin{equation}\label{eqn:AH}
  \left\|AH(z)\right\|\le (M+1)|z^{-1}g(z)|\le c\min(1,|z|^{-\alpha}),\quad z\in\Sigma_{\pi-\theta}.
\end{equation}
By taking $\left\|AH(z)\right\|\le M|z|^{-\alpha}$, we obtain from \eqref{eqn:Asolop}
\begin{equation*}
  \begin{aligned}
    \|AS^{(m)}(t)\|  &\le c\int_{\Gamma} |z|^{m-\al} e^{\Re(z)t} \, |dz| \\
    &\le c\left( \int_{1/t}^\infty r^{m-\al}e^{-rt\cos\theta}  \,dr
  + \int_{-\pi+\theta}^{\pi-\theta}e^{\cos\psi}t^{-m-1+\al}\,d\psi\right)\leq c t^{-m-1+\al}.
  \end{aligned}
\end{equation*}
This shows estimate \eqref{eqn:new1}. To prove estimate \eqref{eqn:new2} with $\nu=0$ we observe that
\begin{equation*}
  \begin{aligned}
    S^{(m)}(t)v & =\frac{1}{2\pi\mathrm{i}}\int_{\Gamma} z^m e^{zt} \frac{g(z)}{z} (g(z)I+A)^{-1} v \,dz \\
    & = \frac{1}{2\pi\mathrm{i}}\int_{\Gamma} z^{m-1}e^{zt} g(z) A^{-1}(g(z)I+A)^{-1} Av dz.
  \end{aligned}
\end{equation*}
Now by noting the identity
\begin{equation*}
g(z)A^{-1}(g(z)I+A)^{-1}= A^{-1}-(g(z)I+A)^{-1}
\end{equation*}
and the fact that $\int_{\Gamma} z^{m-1}e^{zt}\,dz=0$ for $m\ge 1$,
we have
\begin{equation*} 
   \begin{aligned}
      S^{(m)}(t)v & = \frac{1}{2\pi\mathrm{i}}\int_{\Gamma}z^{m-1} e^{zt} v\, dz
      - \frac{1}{2\pi\mathrm{ i}}\int_{\Gamma} z^{m-1}e^{zt}(g(z)I+A)^{-1} \,dzA v\\
      &=-\frac{1}{2\pi\mathrm{i}}\int_{\Gamma} z^{m-1}e^{zt}(g(z)I+A)^{-1} \,dzA v.\\
   \end{aligned}
\end{equation*}
By \eqref{eqn:resolventg} we obtain
\begin{equation*}
  \|(g(z)I+A)^{-1}\| \le M |g(z)|^{-1}=M\bigg|\frac{1+\gamma z^\alpha}{z}\bigg| \leq
  M(|z|^{-1}+\gamma |z|^{\al-1}),
\end{equation*}
and thus using this estimate, we get
\begin{equation*}
  \begin{aligned}
    \|S^{(m)}(t)v\|_{L^2(\Omega)} & \le c\left(\int_{\Gamma} |z|^{m-1}e^{\Re(z)t}\|(g(z)I+A)^{-1}\| \,|dz|\right)\| Av \|_{L^2(\Omega)}\\
		&\le c\left(\int_{1/t}^\infty e^{-rt\cos\theta}(r^{m-2}+\g r^{m+\al-2}) \,dr\right.\\
    &\quad \left.+ \int_{-\pi+\theta}^{\pi-\theta} e^{\cos\psi} (t^{-m+1}+\g t^{-m+1-\al}) \,d\psi \right)
    \| Av \|_{L^2(\Omega)}\\
    &\leq c(t^{-m+1}+\g t^{-m+1-\al})\| Av \|_{L^2(\Omega)}.
  \end{aligned}
\end{equation*}
Since $t^{-m+1}\le T^{\al}t^{-m+1-\al}$ for $t\in(0,T]$, we deduce
\begin{equation*}
   \|S^{(m)}(t)v\|_{L^2(\Omega)} \le c_T t^{-m+1-\al}\| Av \|_{L^2(\Omega)}, \ \ t\in(0,T],
\end{equation*}
with $c_T=c(T^\al+\g)$. Lastly, note that (\ref{eqn:new2}) with $\nu=1$ is equivalent to (\ref{eqn:new1}) with $\nu=0$ and $v$ replaced by $Av$.
\end{proof}

\begin{remark}
We note that this argument is applicable to any sectorial operator $A$, including the Riemann-Liouville
fractional derivative operator in space \cite{JinLazarovPasciakZhou:2013a}.
\end{remark}

Further, the estimates in Theorem~\ref{thm:reg} imply the following result by interpolation.
\begin{remark}\label{thm:reg-init2}
The solution $S(t)v$ to problem \eqref{eqn:rsp} with $f\equiv 0$ satisfies
\begin{equation*}
    \| S^{(m)}(t)v\|_{\dH p} \le ct^{-m-(1-\alpha)(p-q)/2}\| v \|_{\dH q}\quad \forall t\in(0,T],
\end{equation*}
where for $m=0$ and $0\le q\le p \le 2$ or $m>0$ and $0\le p , ~~q\le 2$.
\end{remark}

\subsection{Further discussions on the behavior of the solution}

The estimate \eqref{eqn:new1} holds for any  $t>0$. However, in the case $\nu=1$ and $m=0$
we can improve this estimate for large $t>0$. Namely, if we apply the bound $\|AH(z)\|\le M$ from
\eqref{eqn:AH} in the estimate of \eqref{eqn:Asolop}, we get the following sharper bound for large $t$:
\begin{remark}
For $v\in L^2(\Omega)$ we have the following bound
\begin{equation}\label{eqn:larget}
\|A S(t)v\|_{L^2(\Omega)}\le c t^{-1}\|v\|_{L^2(\Omega)},\ \  t>0,
\end{equation}
which is sharper than \eqref{eqn:new1} for large $t$. This bound together with
\eqref{eqn:new1} with $\nu=0, m=1$, imply the following a priori estimate for the solution of problem \eqref{eqn:rsp}:
\begin{equation*}
\|\partial_t u\|_{L^2(\Omega)}+\| u\|_{\dH 2}+\|\partial_t^\alpha  u\|_{\dH 2}\le ct^{-1}\|v\|_{L^2(\Omega)}
\quad \mbox{for large $t>0$}.
\end{equation*}
\end{remark}

Further, by applying eigenfunction expansion, the solution of the Rayleigh-Stokes problem \eqref{eqn:rsp}
can be written in the form
\begin{equation*}
u(x,t)=\sum_{j=1}^\infty (v,\fy_j) u_j(t)\fy_j(x)+\sum_{j=1}^\infty \left(\int_0^t u_j(t-\tau)f_j(\tau)\,d\tau\right) \fy_j(x),
\end{equation*}
where $f_j(t)=( f(.,t),\fy_j)$ and $u_j(t)$ satisfies the following equation:
\begin{equation}\label{un}
 u'_j(t)+\la_j(1+\g \DD)u_j(t)=0, \ \ u_j(0)=1.
\end{equation}
 To solve (\ref{un}) we apply Laplace transform and use the identities
\begin{eqnarray}
  &&\mathcal{L}\{u'\}(z)=z \mathcal{L}\{ u\}(z)-u(0)\label{Lap}\\
  &&\mathcal{L}\{\DD u\}(z)=z^\al \mathcal{L}\{ u\}(z),\ \ \al\in (0,1),\label{LapD}
\end{eqnarray}
which hold for functions $u(t)$, continuous for $t>0$, and such that $u(0)$ is finite
\cite[equation (1.15)]{MainardiGorenflo:2007}. In this way, for the Laplace transform
of $u_j(t)$, one arrives at
\begin{equation*}
 \mathcal{L} \{u_j\}(z)=\frac{1}{z+\g\la_j z^\al+\la_j}.
\end{equation*}
Based on this representation, in the next theorem we summarize some properties of the
time-dependent components $u_j(t)$, which
are useful in the study of the solution behavior, including the inhomogeneous problem.

Recall that a function $u(t)$ is said to be completely monotone if and only if
\begin{equation*}
  (-1)^n u^{(n)}(t)\ge 0, \mbox{\ for\ all\ }t\ge 0, \ n=0,1,...
\end{equation*}

\begin{theorem} The functions $u_j(t),\ j=1,2,...,$ have the following properties:
	\begin{eqnarray*}
	&&u_j(0)=1,\  0<u_j(t)\le 1, \ t\ge 0,\label{oc1}\\
		&&u_j(t)\ \mbox { are completely\ monotone\ for\ } t\ge 0, \label{cm}\\
		&&|\la_j u_j(t)|\le c\min \{t^{-1},t^{\al-1}\},\  \ t> 0, \ \label{ocHn}\\
	&&\int_0^T |u_j(t)|\,dt<\frac{1}{\la_j},\ \ T>0.\label{intH}	
	\end{eqnarray*}
where the constant $c$ does not depend on $j$ and $t$.
\end{theorem}
\begin{proof}
We introduce the auxiliary functions $v_j(t)$ defined by their Laplace transforms
\begin{equation}\label{Gdef}
\mathcal{L} \{v_j\}(z)=\frac{1+\g\la_j  z^{\al-1}}{z+\g\la_j  z^\al+\la_j}.
\end{equation}
By the property of the Laplace transform $ u(0)=\lim_{z\to +\infty} z \widehat{u}(z)$
we obtain $u_j(0)=1$ and $v_j(0)=1$. Further, taking the inverse Laplace transform of
\eqref{un}, we get
\begin{equation*}
  u_j(t)=\frac{1}{2\pi\mathrm{i}} \int_{Br}e^{zt}\frac{1}{z+\g\la_j z^\al+\la_j}\,dz,
\end{equation*}
where $Br=\{z;\ \Re z=\sigma,\ \sigma>0\}$ is the Bromwich path \cite{Widder:1941}. The function under the
integral has a branch point $0$, so we cut off the negative part of the real axis. Note
that the function $z+\g\la_j  z^\al+\la_j$ has no zero in the main sheet of the Riemann
surface including its boundaries on the cut. Indeed, if $z=\varrho e^{i\theta}$, with
$\varrho>0$, $\theta\in(-\pi,\pi)$, then
\begin{equation*}
  \Im \{z+\g\la_j  z^\al+\la_j\}=\varrho\sin\theta +\g\la_j\varrho^\al\sin\al\theta\neq 0,\quad \theta\neq 0,
\end{equation*}
since $\sin\theta$ and $\sin\al\theta$ have the same sign and $\la_j, \g>0$. Hence,
$u_j(t)$ can be found by bending the Bromwich path into the Hankel path $Ha(\varepsilon)$,
which starts from $-\infty$ along the lower side of the negative real axis, encircles
the disc $|z|=\varepsilon$ counterclockwise and ends at $-\infty$ along the upper side of the
negative real axis. By taking $\varepsilon\to 0$ we obtain
\begin{equation*}
  u_j(t)= \int_0^\infty e^{-rt}K_j(r)\,dr,
\end{equation*}
where
\begin{equation*}
 K_j(r)=\frac{\g}{\pi} \frac{\la_j r^{\al} \sin\al \pi}{(-r+\la_j \g r^\al \cos\al \pi+\la_j)^2+(\la_j \g r^\al \sin\al \pi)^2}.
\end{equation*}
Since $\al\in(0,1)$, and $\la_j,\g>0$, there holds $K_j(r)>0$ for all $r>0$. Hence,
by Bernstein's theorem, $u_j(t)$ are completely monotone functions. In particular,
they are positive and monotonically decreasing. This shows the first two assertions. 

In the same way we prove that the functions $v_j(t)$ are completely monotone and
hence $0<v_j(t)\le 1$. By \eqref{Lap}, and \eqref{Gdef},
$$
\mathcal{L}\{v_j'\}(z) =z\mathcal{L}\{v_j\}(z)-v_j(0)=z\mathcal{L}\{v_j\}(z)-1=-\la_j\mathcal{L}\{u_j\}(z),
$$
which, upon taking the inverse Laplace transform, implies
$ 
  u_j(t)=- v_j'(t)/\la_j.
$ 
Now the third assertion follows by
\begin{equation*}
  \int_0^T |u_j(t)|\,dt =\int_0^T u_j(t)\,dt=-\frac{1}{\la_j}\int_0^T v'_j(t)\,dt=\frac{1}{\la_j}(1-v_j(T))<\frac{1}{\la_j}.
\end{equation*}
Last, using the representation
\begin{equation*}
  u_j(t)=\frac{1}{2\pi\mathrm{i}} \int_{\Gamma}e^{zt}\frac{1}{z+\g\la_j z^\al+\la_j}\,dz=\frac{1}{2\pi\mathrm{i}} \int_{\Gamma}e^{zt}H(z,\la_j)\,dz
\end{equation*}
with
\begin{equation*}
  H(z,\la_j)=\frac{g(z)}{z}(g(z)+\la_j)^{-1},
\end{equation*}
where the function $g(z)$ is defined as in \eqref{eqn:HH}, the last assertion
follows by applying the argument from the proof of Theorem~\ref{thm:reg} with $A$
replaced by $\la_j>0$ and using the following estimate analogous to \eqref{eqn:AH}:
\begin{equation*}
|\la_jH(z,\la_j)|\le M\min(1,|z|^{-\alpha}),\quad z\in\Sigma_{\pi-\theta}.
\end{equation*}
This completes the proof of the proposition.
\end{proof}

By Theorem \ref{thm:reg}, for any $\alpha\in(0,1)$, the solution
operator $S$ has a smoothing property in space of order two. In the limiting case
$\alpha=1$, however, it does not have any smoothing property.
To see this, we consider the eigenfunction expansion:
\begin{equation}\label{eqn:eigexpansion}
u(x,t) = S(t)v = \sum_{j=1}^\infty (v,\fy_j)u_j(t)\fy_j(x).
\end{equation}
In the case $\alpha =1$ we deduce from \eqref{un} and \eqref{Lap}
\begin{equation*}
 \mathcal{L} \{u_j\}(z)=\frac{1+\gamma\lambda_j}{z+\gamma\lambda_jz+\lambda_j}, \quad
 \mbox{which implies} \quad u_j(t)=e^{-\frac{\lambda_j}{1+\gamma\lambda_j}t}.
\end{equation*}
This shows that the problem does not have smoothing property.

\begin{remark}
We observe that if $v\in L^2\II$, then $\| u(t) \|_{\dot H^2\II}$
behaves like $t^{\al-1}$ as $t\rightarrow0$.  This behavior is the identical with
that of the solution to the subdiffusion equation; see \cite[Theorem 4.1]{McLean:2010}
and \cite[Theorem 2.1]{SakamotoYamamoto:2011}. However, as $t\rightarrow \infty$, $\| u(t) \|_{\dot H^2\II}$
 decays like $t^{-1}$, as in the case of standard diffusion equation. The solution
 $u(t)$ of \eqref{eqn:rsp} decays like $t^{-1}$ for $t\rightarrow \infty$. This is faster than $t^{\al-1}$, the decay
of the solution to subdiffusion equation \cite[Corollary 2.6]{SakamotoYamamoto:2011},
but much slower than the exponential decay for the diffusion equation.
\end{remark}

We may extend Theorem \ref{thm:reg} to the case of very weak initial data, i.e., $v\in \dH q$ with $-1<q<0$.
Obviously, for any $t > 0$ the function $u(t) = S(t)v$ satisfies equation \eqref{eqn:rsp} in the sense of 
$\dH q$. Then we appeal to the expansion \eqref{eqn:eigexpansion}. Repeating the argument of Theorem 
\ref{thm:reg} yields $\|  S(t)v-v \|_{\dH q} \le c \| v \|_{\dH q}.$ By Lebesgue's dominated convergence theorem we deduce
\begin{equation*}
  \lim_{t \rightarrow 0^+} \| S(t)-v \|_{\dH q}^2
  = \lim_{t \rightarrow 0^+} \sum_{j=1}^\infty \la_j^q (u_j(t)-1)^2(v,\fy_j)^2=0.
\end{equation*}
Hence, the function $u(t) = S(t)v$ satisfies \eqref{eqn:rsp} and for $t\rightarrow 0$
converges to $v$ in $\dH q$, i.e., $u(t)=S(t)v$ does represent a solution. Further,
the argument of Theorem \ref{thm:reg} yields $u(t)=S(t)v \in \dH {2+q}$ for any $t>0$.

\section{Semidiscrete Galerkin Finite element method}\label{sec:semidiscrete}
In this section we consider the space semidiscrete finite element approximation and
derive optimal error estimates for the homogeneous problem.

\subsection{Semidiscrete Galerkin scheme}
First we recall the $L^2\II$-orthogonal projection $P_h:L^2(\Omega)\to X_h$ and
the Ritz projection $R_h:H^1_0(\Omega)\to X_h$, respectively, defined by
\begin{equation*}
  \begin{aligned}
    (P_h \fy,\chi) & =(\fy,\chi) \quad\forall \chi\in X_h,\\
    (\nabla R_h \fy,\nabla\chi) & =(\nabla \fy,\nabla\chi) \quad \forall \chi\in X_h.
  \end{aligned}
\end{equation*}
For $\fy \in \dH {-s}$ for $0< s \le 1$, the $L^2\II$-projection $P_h$ is not well-defined.
Nonetheless, one may view $(\fy,\chi)$ for $\chi\in X_h \subset \dot H^{s}$
as the duality pairing between the spaces $\dH s$ and $\dH {-s}$ and define $P_h$ in the same manner.

The Ritz projection $R_h$ and the $L^2$-projection $P_h$ have the following properties.
\begin{lemma}\label{lem:prh-bound}
Let the mesh $X_h$ be quasi-uniform. Then the operators $R_h$ and $P_h$ satisfy:
\begin{eqnarray*}
  &   \|R_h \fy-\fy\|_{L^2(\Omega)}+h\|\nabla(R_h \fy-\fy)\|_{L^2(\Omega)}\le ch^q\| \fy\|_{\dot H^q(\Omega)}\quad
   \forall\fy \in \dot H^q(\Omega), \ q=1,2,\\
  & \|P_h \fy-\fy\|_{L^2(\Omega)}+h\|\nabla(P_h \fy-\fy)\|_{L^2(\Omega)}\le ch^q\| \fy\|_{\dot H^q(\Omega)}\quad
   \forall\fy \in \dot H^q(\Omega), \ q=1,2.
\end{eqnarray*}
In addition, $P_h$ is stable on $\dH {q}$ for $-1\le q \le1$.
\end{lemma}

Upon introducing the discrete Laplacian $\Delta_h: X_h\to X_h$ defined by
\begin{equation}\label{eqn:Delh}
  -(\Delta_h\fy,\chi)=(\nabla\fy,\nabla\chi)\quad\forall\fy,\,\chi\in X_h,
\end{equation}
and $f_h= P_h f$, we may write the spatially discrete problem \eqref{eqn:fem} as
to find $u_h \in X_h$ such that
\begin{equation}\label{eqn:femop}
\partial_t u_h - (1 +\g \partial_t^\al)\Delta_h u_h = f_h, \quad u_h(0)=v_h,
\end{equation}
where  $v_h\in X_h$ is a suitable approximation to the initial condition $v$. Accordingly, the
solution operator $S_h(t)$ for the semidiscrete problem \eqref{eqn:fem} is given by
\begin{equation}\label{eqn:disc-solop}
S_h(t)=\frac{1}{2\pi\mathrm{i}}\int_\Gamma e^{z t} H_h(z) \, dz \quad \mbox{with}
\quad   H_h(z)=\frac{g(z)}{z}(g(z)I+A_h)^{-1},
\end{equation}
where $\Gamma$ is the contour defined in \eqref{eqn:Gamma} and
$A_h=-\Delta_h$. Further, with the eigenpairs $\{(\la_j^h,\fy_j^h)\}$ of the discrete
Laplacian $-\Delta_h$, we define the discrete norm $|||\cdot|||_{\dot H^p(\Omega)}$ on the space $X_h$  for
any $p\in\mathbb{R}$
\begin{equation*}
  |||{\fy}|||_{\dot H^p(\Omega)}^2 =
      \sum_{j=1}^N(\la_j^h)^p(\fy,\fy_j^h)^2\quad \forall\fy\in X_h.
\end{equation*}

The stability of the operator $S_h(t)$ is given below. The proof
is similar to that of Theorem \ref{thm:reg}, and hence omitted.
\begin{lemma}\label{lem:discreg}
Let $S_h(t)$ be defined by \eqref{eqn:femop} and $v_h \in X_h$. Then
\begin{equation*}
|||S_h^{(m)}(t)v_h |||_{\dot H^p(\Omega)} \le ct^{-m-(1-\al)(p-q)/2}|||v_h |||_{\dot H^q(\Omega)}, \quad \forall 0<t\le T,
\end{equation*}
where for $m=0$ and $0\le q\le p \le 2$ or $m>0$ and $0\le p , ~~q\le 2$.
\end{lemma}

Now we derive error estimates for the semidiscrete Galerkin scheme \eqref{eqn:femop} using an operator
trick, following the interesting work of Fujita and Suzuki \cite{FujitaSuzuki:1991}.
We note that similar estimates follow also from the technique in \cite{JinLazarovZhou:2013},
but at the expense of an additional logarithmic factor $|\ln h|$ in the case of nonsmooth
initial data.

The following lemma plays a key role in deriving error estimates.
\begin{lemma}\label{lem:keylem}
For any $ \fy\in H_0^1(\Omega)$ and $z\in \Sigma_{\pi-\theta}=\left\{ z: |\arg(z)|\le \pi-\theta \right\}$
for $\theta\in(0,\pi/2)$, there holds
\begin{equation}
    |g(z)| \| \fy \|_{L^2(\Omega)}^2 + \| \nabla \fy \|_{L^2(\Omega)}^2
    \le c\left|g(z)\| \fy \|_{L^2(\Omega)}^2 + (\nabla\fy,\nabla\fy)\right|.
\end{equation}
\end{lemma}
\begin{proof}
By \cite[Lemma 7.1]{FujitaSuzuki:1991}, we have that for any $z\in \Sigma_{\pi-\theta}$
\begin{equation*}
    |z| \| \fy \|_{L^2(\Omega)}^2 + \| \nabla \fy \|_{L^2(\Omega)}^2 \le c\left|z\| \fy \|_{L^2(\Omega)}^2 + (\nabla\fy,\nabla\fy)\right|.
\end{equation*}
Alternatively, it follows from the inequality
\begin{equation*}
  \gamma |z| + \beta \leq \frac{|\gamma z+\beta|}{\sin\frac\theta2}\quad \mbox{for } \gamma,\beta\geq 0, z\in \Sigma_{\pi-\theta},
\end{equation*}
with the choice $\gamma=\|\fy\|_{L^2(\Omega)}^2$ and $\beta=\|\nabla \fy\|_{L^2(\Omega)}^2=(\nabla \fy,\nabla \fy)$.
By Lemma \ref{lem:est-g}, $g(z)\in\Sigma_{\pi-\theta}$ for all $z\in \Sigma_{\pi-\theta}$, and
this completes the proof.
\end{proof}

The next lemma shows an error estimate between $(g(z)I+A)^{-1}v$ and its discrete analogue $(g(z)I+A_h)^{-1}P_h v$.
\begin{lemma}\label{lem:wbound}
Let $ v\in L^2(\Omega) $,  $z\in \Sigma_{\pi-\theta}$,
 $w=(g(z)I+A)^{-1}v$, and $w_h=(g(z)I+A_h)^{-1}P_h v$. Then there holds
\begin{equation}\label{eqn:wboundHa}
    \|  w_h-w \|_{L^2(\Omega)} + h\| \nabla (w_h-w)\|_{L^2(\Omega)}
    \le ch^2 \| v  \|_{L^2(\Omega)}.
\end{equation}
\end{lemma}
\begin{proof}
By the definition, $w$ and $w_h$ respectively satisfy
\begin{equation*}
  \begin{aligned}
    g(z)(w,\chi)+(\nabla w,\nabla \chi)&=(v,\chi),\quad \forall \chi \in V,\\
    g(z)(w_h,\chi)+(\nabla w,\nabla \chi)&=(v,\chi),\quad \forall \chi\in V_h.
  \end{aligned}
\end{equation*}
Subtracting these two identities yields the following orthogonality relation for the error $e=w-w_h$:
\begin{equation}\label{eqn:worthog}
    g(z)(e,\chi) + (\nabla e, \nabla \chi) = 0, \quad \forall\chi\in V_h.
\end{equation}
This and Lemma \ref{lem:keylem} imply that for any $\chi\in V_h$
\begin{equation*}
    \begin{split}
        |g(z)| \| e\|_{L^2(\Omega)}^2  + \| \nabla e \|_{L^2(\Omega)}^2
        & \le c \left|g(z)\| e \|_{L^2(\Omega)}^2 + (\nabla e, \nabla e)\right| \\
        & = c \left|g(z)(e,w-\chi) + (\nabla e, \nabla(w-\chi))\right|.
     \end{split}
\end{equation*}
By taking $\chi=\pi_h w$, the Lagrange interpolant of $w$,
and using the Cauchy-Schwarz inequality, we arrive at
\begin{equation}\label{eqn:control2}
    \begin{aligned}
     |g(z)| \| e \|_{L^2\II}^2 + \| \nabla e \|_{L^2(\Omega)}^2
          & \le c \left(|g(z)| h\|e\|_{L^2\II} \|\nabla w\|_{L^2\II}+
          h \|\nabla e\|_{L^2\II} \| w \|_{\dot H^2 \II} \right).
     \end{aligned}
\end{equation}
Appealing again to Lemma \ref{lem:keylem} with the choice $\fy=w$, we obtain
\begin{equation*}
    |g(z)| \|w \|_{L^2\II}^2 + \|\nabla w\|_{L^2\II}^2
    \le c|((g(z)I+A)w,w)|\le c\| v \|_{L^2\II}\| w\|_{L^2\II}.
\end{equation*}
Consequently
\begin{equation}\label{eqn:wbound2}
    \|w \|_{L^2\II} \le c|g(z)|^{-1}\| v \|_{L^2\II}\quad\mbox{and}\quad
    \|\nabla w \|_{L^2\II} \le c|g(z)|^{-1/2}\| v \|_{L^2\II}.
\end{equation}
In view of \eqref{eqn:wbound2}, a bound on $\| w \|_{\dH 2}$ can be derived
\begin{equation*}
\begin{split}
   \| w \|_{\dot H^{2}\II} & =  \| A w \|_{L^2\II}
   = c\| (-g(z)I+g(z)I+A)(g(z)I+A)^{-1}v \|_{L^2\II}\\
   &\le c\left(\| v \|_{L^2\II}+|g(z)|\|w\|_{L^2\II}\right)\le c\| v\|_{L^2\II}.
\end{split}
\end{equation*}
It follows from this and \eqref{eqn:control2} that
\begin{equation*}
    |g(z)| \| e\|_{L^2\II}^2 + \| \nabla e\|_{L^2\II}^2
     \le ch \| v \|_{L^2\II}\left(|g(z)|^{1/2}\| e\|_{L^2\II}
     + \|\nabla e\|_{L^2\II } \right),
\end{equation*}
and this yields
\begin{equation}\label{eqn:control3}
     |g(z)| \|e\|_{L^2\II}^2 + \|\nabla e\|_{L^2\II}^2\le ch^{2}\| v \|_{L^2\II}^2.
\end{equation}
This gives the desired bound on $\|\nabla e\|_{L^2\II}$. Next, we derive the
estimate on $\|e\|_{L^2\II}$ by a duality argument. For $\fy\in L^2\II$, by setting
$$\psi=(g(z)I+A)^{-1}\fy\quad \text{and}\quad \psi_h=(g(z)I+A_h)^{-1}P_h\fy$$
we have by duality
\begin{equation*}
\|e \|_{L^2\II} \le \sup_{\fy \in L^2\II}\frac{|(e,\fy)|}{\|\fy\|_{L^2\II}}
=\sup_{\fy \in L^2\II}\frac{|g(z)(e,\psi)+(\nabla e,\nabla \psi)|}{\|\fy\|_{L^2\II}}.
\end{equation*}
Then the desired estimate follows from \eqref{eqn:worthog} and \eqref{eqn:control3} by
\begin{equation*}
    \begin{split}
      |g(z)(e,\psi)+(\nabla e,\nabla \psi)|
        & = |g(z)(e,\psi-\psi_h)+(\nabla e,\nabla (\psi-\psi_h))|\\
        & \le |g(z)|^{1/2}\|e\|_{L^2\II} |g(z)|^{1/2}\| \psi-\psi_h \|_{L^2\II}\\
        &\ \ \ \ + \|\nabla e\|_{L^2\II}\| \nabla(\psi-\psi_h) \|_{L^2\II}\\
        & \le ch^{2} \| v \|_{L^2\II}\| \fy \|_{L^2\II}.
    \end{split}
\end{equation*}
This completes proof of the lemma.
\end{proof}

\subsection{Error estimates for the semidiscrete scheme}
Now we can state the error estimate for the nonsmooth initial data $v\in L^2\II$.
\begin{theorem}\label{thm:nonsmooth-initial-op}
Let $u$ and $u_h$ be the solutions of problem \eqref{eqn:rsp} and \eqref{eqn:femop} with $v \in L^2\II$
and $v_h=P_h v$, respectively. Then for $t>0$, there holds:
\begin{equation*}
  \| u(t)-u_h(t) \|_{L^2\II} + h\| \nabla (u(t)-u_h(t)) \|_{L^2\II}
  \le ch^{2} t^{-(1-\al)} \|v\|_{L^2\II}.
\end{equation*}
\end{theorem}
\begin{proof}
The error $e(t):=u(t)-u_h(t)$ can be represented as
\begin{equation*}
e(t)=\frac1{2\pi\mathrm{i}}\int_{\Gamma} e^{zt} \frac{g(z)}{z} (w-w_h) \,dz,
\end{equation*}
with $w=(g(z)I+A)^{-1}v$ and $w_h=(g(z)I+A_h)^{-1}P_h v$.
By Lemma \ref{lem:wbound} and the argument in the proof of Theorem \ref{thm:reg} we have
\begin{equation*}
\| \nabla e(t) \|_{L^2\II}\le ch \| v \|_{L^2\II} \int_{\Gamma} e^{\Re(z)t} \frac{|g(z)|}{|z|} \,|dz|
\le cht^{-(1-\al)} \| v \|_{L^2\II}.
\end{equation*}
A similar argument also yields the $L^2\II$-estimate.
\end{proof}

Next we turn to the case of smooth initial data, i.e., $v\in \dH2$ and $v_h\in R_hv$. We take again
contour $\Gamma=\Gamma_{1/t,\pi-\theta}$. Then the error $e(t)=u(t)-u_h(t)$ can be represented as
\begin{equation*}
e(t)=\frac1{2\pi\mathrm{i}}\int_{\Gamma} e^{zt} \frac{g(z)}{z} \left((g(z)I+A)^{-1}-(g(z)I+A_h)^{-1}R_h\right)v \,dz.
\end{equation*}
By the equality
\begin{equation*}
   \frac{g(z)}{z} (g(z)I+A)^{-1} = z^{-1}I-z^{-1}(g(z)I+A)^{-1}A,
\end{equation*}
we can obtain
\begin{equation}\label{eqn:smooth-er-rep}
  \begin{split}
    e(t)=\frac1{2\pi\mathrm{i}}\left(\int_{\Gamma} e^{zt} z^{-1} (w_h(z)-w(z)) \,dz
    + \int_{\Gamma} e^{zt} z^{-1} (v-R_hv) \,dz\right),
  \end{split}
\end{equation}
where $w(z)=(g(z)I+A)^{-1}A v$ and $w_h(z)=(g(z)I+A_h)^{-1}A_hR_hv$. Then we derive the following error estimate.

\begin{theorem}\label{thm:smooth-initial-op}
Let $u$ and $u_h$ be the solutions of problem \eqref{eqn:rsp} and \eqref{eqn:femop} with $v \in \dH2$
and $v_h=R_h v$, respectively. Then for $t>0$, there holds:
\begin{equation}\label{eqn:smooth-initial-op}
  \| u(t)-u_h(t) \|_{L^2\II} + h\| \nabla (u(t)-u_h(t)) \|_{L^2\II}
  \le c h^{2} \|v\|_{\dH2}.
\end{equation}
\end{theorem}
\begin{proof}
Let $w(z)=(g(z)I+A)^{-1}Av$ and $w_h(z)=(g(z)I+A_h)^{-1}A_hR_hv$.
Then Lemmas \ref{lem:prh-bound} and \ref{lem:wbound}, and the identity $A_hR_h=P_h A$ give
\begin{equation*}
  \| w(z)-w_h(z) \|_{L^2\II} + h\| \nabla (w(z)-w_h(z)) \|_{L^2\II}
  \le c h^{2} \| Av \|_{L^2\II}.
\end{equation*}
Now it follows from this and the representation \eqref{eqn:smooth-er-rep} that
\begin{equation*}
\begin{split}
  \|e(t)\| &\le c h^2 \| Av \|_{L^2\II}\int_{\Gamma} e^{\Re(z)t}| z|^{-1} \,|dz|\\
      &\le  c h^2 \| Av \|_{L^2\II}\left(\int_{1/t}^\infty e^{-rt\cos\theta} r^{-1} \,dr
      + \int_{-\pi+\theta}^{\pi-\theta} e^{\cos\psi} \,d\psi \right)\\
      &\le c h^2 \| Av \|_{L^2\II}=ch^2 \| v \|_{\dH2}.
\end{split}
\end{equation*}
Hence we obtain the $L^2\II$-error estimate. The $H^1\II$-error estimate follows analogously.
\end{proof}

\begin{remark}\label{rem:smoothPh}
For smooth initial data $v\in \dH 2$, we may also take the approximation $v_h=P_hv$. Then the error can be split into
\begin{equation*}
e(t)=S(t)v-S_h(t)P_hv=(S(t)v-S_h(t)R_hv)+(S_h(t)R_hv-S_h(t)P_hv).
\end{equation*}
Theorem \ref{thm:smooth-initial-op} gives an estimate of the first term. A bound for
the second term follows from Lemmas \ref{lem:prh-bound} and \ref{lem:discreg}
\begin{equation*}
\| S_h(t)(P_hv-R_hv) \|_{\dH p} \le c\| P_hv-R_hv \|_{\dH p} \le ch^{2-p}\|v\|_{\dH2}.
\end{equation*}
Thus the error estimate \eqref{eqn:smooth-initial-op} holds for the initial approximation $v_h=P_hv$.
It follows from this, Theorem \ref{thm:nonsmooth-initial-op},and interpolation that
for all $q\in [0,2]$ and $v_h=P_hv$, there holds
\begin{equation*}
  \| u(t)-u_h(t) \|_{L^2\II} + h\| \nabla (u(t)-u_h(t)) \|_{L^2\II}
   \le c h^{2} t^{-(1-\al)(2-q)/2} \|v\|_{\dot H^q \II}.
\end{equation*}
\end{remark}

\begin{remark}\label{rem:veryweak}
If the initial data is very weak, i.e., $v\in\dot H^q\II$ with $-1<q<0$, Then the argument of
\cite[Theorem 2]{JinLazarovPasciakZhou:2013} yields the following optimal error estimate for
the semidiscrete finite element approximation \eqref{eqn:fem}
\begin{equation}\label{eqn:weaksemi}
    \| u(t)-u_h(t) \|_{L^2\II} + h\| \nabla (u(t)-u_h(t)) \|_{L^2\II}
  \le c h^{2+q} t^{-(1-\al)} \|v\|_{\dH q}.
\end{equation}
\end{remark}

\section{Fully discrete schemes}\label{sec:fullydiscrete}
Now we develop two fully discrete schemes for problem \eqref{eqn:rsp} based on convolution quadrature
(see \cite{Lubich:1988,LubichSloanThomee:1996,Lubich:2004,CuestaLubichPlencia:2006} for detailed discussions),
and derive optimal error estimates for both smooth and nonsmooth initial data.

\subsection{Convolution quadrature}
First we briefly describe the abstract framework in \cite[Sections 2 and 3]{CuestaLubichPlencia:2006},
which is instrumental in the development and analysis of fully discrete schemes.
Let $K$ be a complex valued or operator valued function that is analytic in a sector $\Sigma_{\pi-\theta}$,
$\theta\in(0,\pi/2)$ and is bounded by
\begin{equation}\label{eqn:Ksect}
  \|K(z)\|\leq M|z|^{-\mu}\quad \forall z\in\Sigma_{\pi-\theta},
\end{equation}
for some real numbers $\mu$ and $M$. Then $K(z)$ is the Laplace transform of a distribution $k$ on
the real line, which vanishes for $t<0$, has its singular support empty or concentrated at $t=0$,
and which is an analytic function for $t>0$. For $t>0$, the analytic function $k(t)$ is given by the inversion formula
\begin{equation*}
  k(t) = \frac{1}{2\pi\mathrm{i}}\int_\Gamma K(z)e^{zt}dz, \ \ t>0,
\end{equation*}
where $\Gamma$ is a contour lying in the sector of analyticity, parallel to its boundary and oriented with
increasing imaginary part. With $\partial_t$ being time differentiation, we define $K(\partial_t)$ as the
operator of (distributional) convolution with the kernel $k:K(\partial_t)g=k\ast g$ for a function $g(t)$
with suitable smoothness.

A convolution quadrature approximates $K(\partial_t)g(t)$ by a discrete
convolution $K(\bPtau) g(t)$. Specifically, we divide the time interval $[0,T]$ into
$N$ equal subintervals with a time step size $\tau=T/N$, and define the approximation:
\begin{equation*}
  K(\bPtau) g(t) = \sum_{0\leq j\tau\leq t}\omega_jg(t-j\tau), \ \ t>0,
\end{equation*}
where the quadrature weights $\{\omega_j\}_{j=0}^\infty$ are determined by the generating function
\begin{equation*}
  \sum_{j=0}^\infty \omega_j\xi^j = K(\delta(\xi)/\tau).
\end{equation*}
Here $\delta$ is the quotient of the generating polynomials of
a stable and consistent linear multistep method. In this work, we consider
the backward Euler (BE) method and second-order backward difference (SBD) method, for which
\begin{equation*}
  \delta(\xi) = \left\{\begin{aligned}
    (1-\xi), \qquad &\ \  \mbox{BE},\\
    (1-\xi) + (1-\xi)^2/2, &\ \ \mbox{SBD}.
  \end{aligned}\right.
\end{equation*}

Now we specialize the construction to the semidiscrete problem \eqref{eqn:femop}.
By integrating \eqref{eqn:femop} from $0$ to $t$, we arrive at
a representation of the semidiscrete solution $u_h$
\begin{equation*}
  u_h + (\g\partial_t^{\al-1}+\partial_t^{-1})A_h u_h =v_h + \partial_t^{-1}f_h.
\end{equation*}
where $\partial_t^{\beta} u$, $\beta<0$, denotes the Riemann-Liouville integral $\partial_t^\beta u=
\frac{1}{\Gamma(-\beta)}\int_0^t(t-s)^{-\beta-1}u(s)ds$. The left-hand side is a convolution, which
we approximate at $t_n=n\tau$ with $U_h^n$ by
\begin{equation*}
  U_h^n + (\g\bPtau^{\alpha-1} + \bPtau^{-1})A_hU_h^n = v_h + \bPtau^{-1}f_h,
\end{equation*}
where the symbols $\bPtau^{\alpha-1}$ and $\bPtau^{-1}$ refer to relevant convolution quadrature generated by 
the respective linear multistep method. For the convenience of numerical implementation, we rewrite them in a
time stepping form.

\paragraph{\bf 4.1.1 The backward Euler (BE) method}
 The BE method is given by: Find $U_h^n$ for $n=1,2,\ldots,N$ such that
\begin{equation}\label{eqn:BE-v1}
  U_h^n + (\g\bPtau^{\alpha-1} + \bPtau^{-1})A_hU_h^n = v_h + \bPtau^{-1} f_h(t_n)
\end{equation}
with the convolution quadratures $\bPtau^{\alpha-1}$ and $\bPtau^{-1}$ generated by
the BE method. By applying $\bPtau$ to the scheme
\eqref{eqn:BE-v1} and the associativity of convolution, we deduce that it can be
rewritten as: with $U_h^0=v_h\in X_h$ and $F_h^n=f_h(t_n)$, find $U_h^n$ for $n=1,2,...,N$ such that
\begin{equation}\label{eqn:BE}
    \tau^{-1}\left(U_h^n-U_h^{n-1}\right)+ \gamma \bPtau^\alpha(A_hU_h^n) + A_h U_h^n = F_h^n.
\end{equation}
\begin{remark}
In the scheme \eqref{eqn:BE}, the term at $n=0$ in $\bPtau A_hU_h^n$ can be omitted without affecting its convergence rate
\cite{Sanz-Serna:1988,LubichSloanThomee:1996}.
\end{remark}

\paragraph{\bf 4.1.2 The second-order backward difference (SBD) method}
Now we turn to the SBD scheme. It is known that it is only first-order accurate if $g(0)\neq0$, e.g.,
for $g\equiv1$ \cite[Theorem 5.1]{Lubich:1988} \cite[Section 3]{CuestaLubichPlencia:2006}. The
first-order convergence is numerically also observed on problem \eqref{eqn:rsp}. Hence, one needs to
correct the scheme, and we follow the approach proposed in \cite{LubichSloanThomee:1996,
CuestaLubichPlencia:2006}. Using the identity
\begin{equation*}
  (I+(\partial^{\alpha-1}_t +\partial^{-1}_t)A_h)^{-1}=I-(I+(\partial^{\alpha-1}_t+\partial^{-1}_t)
  A_h)^{-1}(\partial^{\alpha-1}_t+\partial^{-1}_t)A_h,
\end{equation*}
we can rewrite the semidiscrete solution $u_h$ into
\begin{equation*}
  u_h=v_h + (I+(\g\partial_t^{\al-1}+\partial_t^{-1})A_h)^{-1}(-(\g\partial_t^{\al-1} + \partial_t^{-1})A_hv_h + \partial^{-1}_tf_{h,0} + \partial_t^{-1}\tilde{f}_h),
\end{equation*}
where $f_{h,0} = f_h(0)$ and $\tilde{f}_h=f_h-f_h(0)$. This leads to the convolution quadrature
\begin{equation}\label{eqn:SBD-v1}
  \begin{aligned}
    U_h^n & = v_h+(I+(\g\bPtau^{\al-1}+\bPtau^{-1})A_h)^{-1}
     (-(\g\bPtau^{\al}\partial_t^{-1} + \partial_t^{-1})A_hv_h\\
      & \qquad +\partial_t^{-1}f_{h,0}(t_n) + \bPtau^{-1} \tilde{f}_h(t_n)).
  \end{aligned}
\end{equation}
The purpose of keeping the operator $\partial_t^{-1}$ intact in \eqref{eqn:SBD-v1} is to achieve
a second-order accuracy, cf. Lemma \ref{lem:boundSBD} below. Letting $1_\tau=(0,3/2,1,\ldots)$, and noting
the identity $1_\tau=\bPtau\partial^{-1} 1$ at grid points $t_n$, and associativity of
convolution, \eqref{eqn:SBD-v1} can be rewritten as
\begin{equation*}
  (I+(\gamma\bPtau^{\alpha-1} +\bPtau^{-1})A_h)(U_h^n - v_h) =
  -(\gamma\bPtau^{\alpha-1}+\bPtau^{-1})A_h1_\tau v_h + \bPtau^{-1} 1_\tau f_{h,0}(t_n) + \bPtau^{-1} \tilde{f}_h(t_n) .
\end{equation*}
Next by applying the operator $\bPtau$, we obtain
\begin{equation}\label{eqn:SBD}
  \bPtau(U_h^n-v_h) + (\gamma\bPtau^{\alpha} + I)A_h(U_h^n-v_h) = - (\gamma\bPtau^{\alpha}+I)A_h1_\tau v_h + 1_\tau f_{h,0}(t_n) + \tilde{f}_h(t_n).
\end{equation}
Thus we arrive at  a time stepping scheme: with $U_h^0=v_h$, find $U_h^n$ such that
\begin{equation*}
  \tau^{-1}{\left(3U_h^1/2-3U_h^0/2\right)} + \gamma \tilde{\partial}_\tau^\alpha A_hU_h^1 + A_hU_h^1+ A_hU_h^0/2 = F_h^1 + F_h^0/2,
\end{equation*}
and for $n\geq 2$
\begin{equation*}
    \bPtau U_h^n +\gamma \tilde{\partial}_\tau^\alpha A_hU_h^n + A_hU_h^n = F_h^n,
\end{equation*}
where the convolution quadrature $\tilde{\partial}_\tau^\alpha \varphi^n $ is given by
$$\tilde{\partial}_\tau^\alpha \varphi^n = \tau^{-\al}(\sum_{j=1}^n\omega^\alpha_{n-j}\varphi^j
+ \omega^\alpha_{n-1}\varphi^0/2),$$
with the weights $\{\omega_j^\alpha\}$ generated by the SBD method.

The error analysis of the fully discrete schemes \eqref{eqn:BE} and \eqref{eqn:SBD} for the case $f\equiv0$
will be carried out below, following the general strategy in \cite[Section 4]{CuestaLubichPlencia:2006}.
\subsection{Error analysis of the backward Euler method}
Upon recalling the function $g(z)$ from \eqref{eqn:HH} and denoting
\begin{equation}\label{eqn:G1}
 G(z)=(I+g(z)^{-1}A_h)^{-1},
\end{equation}
we can write the difference between $u_h(t_n)$ and $U_h^n$ as
\begin{equation}\label{eqn:diff1}
 U_h^n-u_h(t_n)= (G(\bPtau)-G(\partial_t))v_h.
\end{equation}

For the error analysis, we need the following estimate \cite[Theorem 5.2]{Lubich:1988}.
\begin{lemma}\label{lem:boundBE}
Let $K(z)$ be analytic in $\Sigma_{\pi-\theta}$ and \eqref{eqn:Ksect} hold.
Then for $g(t)=ct^{\beta-1}$, the convolution quadrature based on the
BE satisfies
\begin{equation*}
    \| (K(\partial_t)-K(\bPtau))g(t)  \| \le
    \left\{ \begin{array}{ll}
     ct^{\mu-1}\tau^\beta, &~~0<\beta\le 1, \\
     ct^{\mu+\beta-2}\tau, &~~\beta\ge1.\end{array}\right.
\end{equation*}
\end{lemma}

Now we can state the error estimate for nonsmooth initial data $v\in L^2(\Omega)$.
\begin{lemma}\label{lem:timenonsmooth}
Let $u_h$ and $U_h^n$ be the solutions of problem \eqref{eqn:femop} and \eqref{eqn:BE} with
$v\in L^2(\Omega)$, $U_h^0= v_h = P_hv$ and $f\equiv0$, respectively. Then there holds
\begin{equation*}
   \| u_h(t_n)-U_h^n \|_{L^2(\Omega)} \le c \tau t_n^{-1}  \| v\|_{L^2\II}.
\end{equation*}
\end{lemma}
\begin{proof}
By \eqref{eqn:resolvent} and the identity $G(z) = g(z)(g(z)I+A_h)^{-1}$ for
$z\in \Sigma_{\pi-\theta}$, there holds
\begin{equation*}
    \| G(z)\| \le c \quad \forall z\in \Sigma_{\pi-\theta}.
\end{equation*}
Then \eqref{eqn:diff1} and Lemma \ref{lem:boundBE} (with $\mu=0$ and $\beta=1$) give
\begin{equation*}
  \| U_h^n-u_h(t_n) \|_{L^2(\Omega)} \le c \tau t_n^{-1} \| v_h \|_{L^2(\Omega)},
\end{equation*}
and the desired result follows directly from the $L^2(\Omega)$ stability of $P_h$.
\end{proof}

Next we turn to smooth initial data, i.e., $v\in \dH 2$.
\begin{lemma}\label{lem:timesmooth}
Let $u_h$ and $U_h^n$ be the solutions of problem \eqref{eqn:femop} and \eqref{eqn:BE} with
$v\in \dH 2$, $U_h^0= v_h=R_hv$ and $f\equiv0$, respectively. Then there holds
\begin{equation*}
   \| u_h(t_n)-U_h^n \|_{L^2(\Omega)} \le c \tau t_n^{-\al} \| Av \|_{L^2\II}.
\end{equation*}
\end{lemma}
\begin{proof}
With the identity
\begin{equation*}
  A_h^{-1}(I+g(z)^{-1}A_h)^{-1}= A_h^{-1}-(g(z)I+A_h)^{-1},
\end{equation*}
and denoting $G_s(z)= -(g(z)I+A_h)^{-1}$, the error $U_h^n-u_h(t_n)$ can be represented by
\begin{equation*}
 U_h^n-u_h(t_n)= (G_s(\bPtau)-G_s(\partial_t))A_h v_h.
\end{equation*}
From \eqref{eqn:resolvent} and Lemma \ref{lem:est-g} we deduce
\begin{equation*}
    \| G_s(z)\| \le  M |g(z)|^{-1}=M\bigg|\frac{1+\gamma z^\alpha}{z}\bigg| \leq
  M(|z|^{-1}+\gamma |z|^{\al-1})\quad \forall z\in \Sigma_{\pi-\theta}.
\end{equation*}
Now Lemma \ref{lem:boundBE} (with $\mu=1-\al$ and $\beta=1$) gives
\begin{equation*}
  \|U_h^n-u_h(t_n) \|_{L^2(\Omega)} \le c \tau t_n^{-\al} \| A_h v_h \|_{L^2(\Omega)},
\end{equation*}
and the desired estimate follows directly from the identity $A_hR_h=P_hA$.
\end{proof}

\begin{remark}
By Lemma \ref{lem:timesmooth}, the error estimate exhibits a singular behavior of order $t^{-\alpha}$
as $t\rightarrow 0^+$, even for smooth initial data $v\in\dH2$. Nonetheless, as $\al \rightarrow 0^+$,
problem \eqref{eqn:rsp} reduces to the standard parabolic equation, and accordingly the singular behavior
disappears for smooth data, which coincides with the parabolic counterpart \cite{Thomee:2006}.
\end{remark}

Now we can state error estimates for the fully discrete scheme \eqref{eqn:BE}
with smooth and nonsmooth initial data, by the triangle inequality, Theorems
\ref{thm:nonsmooth-initial-op} and \ref{thm:smooth-initial-op},
Lemmas \ref{lem:timenonsmooth} and
\ref{lem:timesmooth}, respectively for the nonsmooth and smooth initial data.

\begin{theorem}\label{thm:error_fully_smooth}
Let $u$ and $U_h^n$ be the solutions of problem \eqref{eqn:rsp} and \eqref{eqn:BE} with
$U_h^0= v_h$ and $f\equiv0$, respectively. Then the following estimates hold.
\begin{itemize}
  \item[(a)] If $v\in \dH 2$ and $v_h=R_h v$, then
  \begin{equation*}
   \| u(t_n)-U_h^n \|_{L^2(\Omega)} \le c (\tau t_n^{-\al} + h^2)  \| v \|_{\dH 2}.
  \end{equation*}
  \item[(b)] If $v\in L^2(\Omega)$ and $v_h=P_hv$, then
  \begin{equation*}
   \| u(t_n)-U_h^n \|_{L^2(\Omega)} \le c (\tau t_n^{-1} + h^2 t_n^{\al-1})  \| v\|_{L^2\II}.
  \end{equation*}
\end{itemize}
\end{theorem}

\begin{remark}\label{rem:int_fully1}
For $v\in \dH 2$, we can also choose $v_h=P_hv$. Let $\overline{U}^n_h$ be the corresponding
solution of the fully discrete scheme with $v_h=P_hv$. By the stability of the scheme,
a direct consequence of Lemma \ref{lem:timesmooth}, we have
\begin{equation*}
  \|U_h^n-\overline{U}^n_h\|_{L^2(\Omega)}\le c\| R_h v- P_h v \|_{L^2(\Omega)}\le ch^2\| v \|_{\dH 2}.
\end{equation*}
Thus the estimate in Theorem \ref{thm:error_fully_smooth}(a) still holds for $v_h=P_hv$.
Then by interpolation with the estimate for $v\in L^2(\Omega)$, we deduce
\begin{equation*}
   \| u(t_n)-U_h^n \|_{L^2(\Omega)} \le c (\tau t_n^{-1+(1-\al)q/2} +h^{2} t_n^{-(1-\al)(2-q)/2}) \| v \|_{\dot H^q \II},
   \quad 0 \le q \le 2.
\end{equation*}
\end{remark}
\begin{remark}\label{rem:weak1}
In case of very weak initial data, i.e., $v\in \dH q$ with $-1<q<0$,
by Lemma \ref{lem:timenonsmooth}, the inverse inequality \cite[pp. 140]{Ciarlet:2002}
and Lemma \ref{lem:prh-bound}
we have
\begin{equation*}
 \begin{aligned}
   \| u_h(t_n) - U_h^n  \|_{L^2\II} & \le c\tau t_n^{-1} \| P_h v \|_{L^2\II}
   \le c \tau h^{q} t_n^{-1} \| P_h v \|_{\dH q}
   \le c \tau h^q t_n^{-1} \|v \|_{\dH q}.
 \end{aligned}
\end{equation*}
This and Remark \ref{rem:veryweak} yield the following error estimate 
\begin{equation*}
\| u(t_n) - U_h^n  \|_{L^2\II} \le c (\tau h^q t_n^{-1} + h^{2+q}t_n^{\al-1}) \|v \|_{\dH q}.
\end{equation*}
\end{remark}

\subsection{Error analysis of the second-order backward difference method}
With $G(z)=-g(z)^{-1}z(I+g(z)^{-1}A_h)^{-1}A_h=-zA_h(g(z)I+A_h)^{-1}$, we have
\begin{equation}\label{eqn:Gsbd}
    u_h-U_h^n=(G(\partial_t)-G(\bPtau))\partial_t^{-1}v_h.
\end{equation}

Like Lemma \ref{lem:boundBE}, the following estimate holds (see \cite[Theorem 5.2]{Lubich:1988} \cite[Theorem 2.2]{Lubich:2004}).
\begin{lemma}\label{lem:boundSBD}
Let $K(z)$ be analytic in $\Sigma_{\pi-\theta}$ and \eqref{eqn:Ksect} hold. Then for
$g(t)=ct^{\beta-1}$, the convolution quadrature based on the
SBD satisfies
\begin{equation*}
    \| (K(\partial_t)-K(\bPtau)g(t)  \| \le
    \left\{ \begin{array}{ll}
     ct^{\mu-1}\tau^\beta, &~~0<\beta\le 2, \\
     ct^{\mu+\beta-3}\tau^2, &~~\beta\ge2.\end{array}\right.
\end{equation*}
\end{lemma}

Now we can state the error estimate for nonsmooth initial data $v\in L^2(\Omega)$.
\begin{lemma}\label{lem:timenonsmooth2}
Let $u_h$ and $U_h^n$ be the solutions of problem \eqref{eqn:femop} and \eqref{eqn:SBD} with
$v\in L^2(\Omega)$, $U_h^0= v_h = P_hv$ and $f\equiv0$, respectively. Then there holds
\begin{equation*}
   \| u_h(t_n)-U_h^n \|_{L^2(\Omega)} \le c \tau^2 t_n^{-2}  \| v\|_{L^2\II}.
\end{equation*}
\end{lemma}
\begin{proof}
By \eqref{eqn:resolvent} and the identity
\begin{equation*}
   G(z) = -z A_h(g(z)I+A_h)^{-1}=-z(I-g(z)(g(z)I+A_h)^{-1})\quad \forall z\in \Sigma_{\pi-\theta},
\end{equation*}
there holds
\begin{equation*}
    \| G(z)\| \le c|z|, \quad \forall z\in \Sigma_{\pi-\theta}.
\end{equation*}
Then \eqref{eqn:Gsbd} and Lemma \ref{lem:boundSBD} (with $\mu=-1$ and $\beta=2$) give
\begin{equation*}
  \| U_h^n-u_h(t_n) \|_{L^2(\Omega)} \le c\tau^2 t_n^{-2} \| v_h \|_{L^2(\Omega)},
\end{equation*}
and the desired result follows directly from the $L^2(\Omega)$ stability of $P_h$.
\end{proof}

Next we turn to smooth initial data $v\in \dH 2$.
\begin{lemma}\label{lem:timesmooth2}
Let $u_h$ and $U_h^n$ be the solutions of problem \eqref{eqn:femop} and \eqref{eqn:SBD} with
$v\in \dH 2$, $U_h^0= v_h=R_hv$ and $f\equiv0$, respectively. Then there holds
\begin{equation*}
   \| u_h(t_n)-U_h^n \|_{L^2(\Omega)} \le c \tau^2 t_n^{-1-\al} \| Av \|_{L^2\II}.
\end{equation*}
\end{lemma}
\begin{proof}
By setting $G_s(z)= -z(g(z)I+A_h)^{-1}$, $U_h^n-u_h(t_n)$ can be represented by
\begin{equation*}
 U_h^n-u_h(t_n)= (G_s(\bPtau)-G_s(\partial_t))A_h v_h.
\end{equation*}
From \eqref{eqn:resolvent} and Lemma \ref{lem:est-g} we deduce
\begin{equation*}
    \| G_s(z)\| \le  M |z||g(z)|^{-1}\le (1+\gamma |z|^\al), \quad \forall z\in \Sigma_{\pi-\theta}.
\end{equation*}
Now Lemma \ref{lem:boundSBD} (with $\mu=-\al$ and $\beta=2$) gives
\begin{equation*}
  \|U_h^n-u_h(t_n)\|_{L^2(\Omega)} \le c\tau^2 t_n^{-1-\al} \| A_h v_h \|_{L^2(\Omega)},
\end{equation*}
and the desired estimate follows from the identity $A_hR_h=P_hA$.
\end{proof}

Then we have the following error estimates for the fully discrete scheme \eqref{eqn:SBD}.
\begin{theorem}\label{thm:error_fully_smooth2}
Let $u$ and $U_h^n$ be solutions of problem \eqref{eqn:rsp} and \eqref{eqn:SBD} with
$U_h^0$ and $f\equiv0$, respectively. Then the following error estimates hold.
\begin{itemize}
  \item[(a)] If $v\in \dH 2$, and $U_h^0=R_hv$, there holds
   \begin{equation*}
     \| u(t_n)-U_h^n \|_{L^2(\Omega)} \le c(\tau^2 t_n^{-1-\al} + h^2)  \|v\|_{\dH2}.
   \end{equation*}
  \item[(b)] If $v\in L^2(\Omega)$, and $U_h^0=P_hv$, there holds
   \begin{equation*}
      \| u(t_n)-U_h^n \|_{L^2(\Omega)} \le c(\tau^2 t_n^{-2} + h^2 t_n^{\al-1})  \|v\|_{L^2\II}.
   \end{equation*}
\end{itemize}
\end{theorem}

\begin{remark}\label{rem:int_fully2}
By the stability of the scheme, a direct consequence of Lemma \ref{lem:timesmooth2},
and the argument in Remark \ref{rem:int_fully1}, the estimate in Theorem
\ref{thm:error_fully_smooth2}(a) still holds for $v_h=P_hv$. Then by interpolation we have
\begin{equation*}
   \| u(t_n)-U_h^n \|_{L^2(\Omega)} \le c(\tau^2 t_n^{-2+(1-\al)q/2} +h^{2} t^{-(1-\al)(2-q)/2}) \|v\|_{\dot H^q \II},
     \quad 0 \le q \le 2.
\end{equation*}
\end{remark}

\begin{remark}\label{rem:weak2}
In case of very weak initial data $v\in \dH q$, $-1<q<0$, the argument in Remark \ref{rem:weak1} yields
\begin{equation*}
 \| u(t_n) - U_h^n  \|_{L^2\II} \le c(\tau^2 h^q t_n^{-2} + h^{2+q}t^{\al-1}) \|v \|_{\dH q}.
\end{equation*}
\end{remark}

\section{Numerical results}\label{sec:numeric}
In this part, we present numerical results to verify the convergence theory in Sections \ref{sec:semidiscrete}
and \ref{sec:fullydiscrete}. We shall consider one- and two-dimensional examples with smooth,
nonsmooth and very weak initial data. In the one-dimensional case, we take $\Omega=(0,1)$, and in the
two-dimensional case $\Omega=(0,1)^2$.
Here we use the notation $\chi_S$ for  the characteristic function of the set $S$.
The following four cases are considered.
\begin{itemize}
\item[(a)] smooth: $v=\sin(2\pi x)$ which is in $H^2\II\cap H_0^1\II$.
\item[(b)] nonsmooth: $v=\chi_{(0,1/2]}$;
the jump at $x=1/2$ and $v(0) \not = 0$ lead to $v \notin \dot H^1(\Omega)$; but
for any $\epsilon\in(0,1/2)$, $v\in \dot H^{{1/2}-\epsilon}\II$.
\item[(c)] very weak data: $v=\delta_{1/2}(x)$ which is a Dirac $\delta$-function
concentrated at $x=0.5$. By Sobolev imbedding theorem, $v\in \dH {-1/2-\epsilon}$ for $\epsilon>0$.
\item[(d)] two-dimensional example: $v=\chi_{(0,1/2]\times(0,1)}$ which is in  $\dot H^{{1/2}-\epsilon}\II$
for any $\epsilon>0$.
\end{itemize}

In our experiments, we fix the parameter $\g=1$ in \eqref{eqn:rsp} for all cases.
We examine separately the spatial and temporal convergence rates at $t=0.1$.
For the case of nonsmooth initial data, we are especially interested in the errors
for $t$ close to zero.
The exact solutions to these examples can be expressed in terms of generalized
Mittag-Leffler functions, which however is difficult to
compute, and hence we compute the reference solution on a very refined mesh.
We report   
the normalized errors $\| e^n \|_{L^2\II}/\| v
\|_{L^2\II}$ and $\| e^n \|_{\dot H^1\II}/\| v \|_{L^2\II}$,
$e^n = u(t_n)-U_h^n$,
for both smooth and nonsmooth data.

In our computation, we divide the unit interval $(0,1)$ into $K=2^k$ equally spaced subintervals, with a
mesh size $h=1/K$. The finite element space $X_h$ consists of continuous piecewise linear functions.
Similarly, we take the uniform temporal mesh with a time step size $\tau=t/N$, with $t$ being the time
of interest.

\subsection{Numerical results for example (a)}
First, we fix the mesh size $h$ at
$h = 2^{-11}$ so that the error incurred by spatial discretization is negligible, which enable us to
examine the temporal convergence rate. In Table \ref{tab:smooth_time}, we show the $L^2\II$-norm of the
error at $t=0.1$ for different $\al$ values. In the table, BE and SBD denote the backward Euler method
and the second-order backward difference method, respectively, \texttt{rate} refers to the empirical
convergence rate when the time step size $\tau$ (or the mesh size $h$) halves, and the numbers in the bracket denote
theoretical convergence rates. In Figure \ref{fig:smooth_time} we plot the results for $\al=0.5$ in a
log-log scale. A convergence rate of order $O(\tau)$ and $O(\tau^2)$ is observed for the BE method and
the SBD method, respectively, which agrees well with our convergence theory. Further, we observe that
the error decreases as the fractional order $\alpha$ increases.

\begin{table}[htb!]
\caption{The $L^2\II$-norm of the error for example (a): $t=0.1$ and $h=2^{-11}$.}
\label{tab:smooth_time}
\begin{center}
\vspace{-.3cm}
     \begin{tabular}{|c|c|ccccc|c|}
     \hline
     & $\tau$ & $1/5$ & $1/10$ &$1/20$ &$1/40$ & $1/80$  &rate \\
     \hline
     BE
     & $\al=0.1$     &6.75e-3 &2.42e-3 &1.00e-3 &4.55e-4 &2.15e-4   &$\approx$ 1.15 (1.00) \\
     \cline{2-7}
     & $\al=0.5$     &3.68e-3  &1.73e-3 &8.42e-4 &4.13e-4 &2.03e-4   &$\approx$ 1.04 (1.00) \\
     \cline{2-7}
     & $\al=0.9$      &4.12e-4 &2.03e-4 &1.00e-4 &4.96e-5 &2.43e-5   &$\approx$ 1.03 (1.00) \\
     \hline
     SBD
     & $\al=0.1$      &5.59e-3 &4.82e-4 &1.18e-4 &2.77e-5 &6.66e-6   &$\approx$ 2.06 (2.00)\\
     \cline{2-7}
     & $\al=0.5$      &1.05e-3 &2.39e-4  &5.33e-5 &1.28e-5 &3.14e-6   &$\approx$ 2.08 (2.00) \\
     \cline{2-7}
     & $\al=0.9$      &7.62e-5  &1.64e-5 &3.86e-6 &9.48e-7 &2.46e-7  &$\approx$ 2.06 (2.00) \\
     \hline
     \end{tabular}
\end{center}
\end{table}

\begin{figure}[htb!]
  \includegraphics[trim = 1cm .1cm 2cm 0.0cm, clip=true,width=8cm]{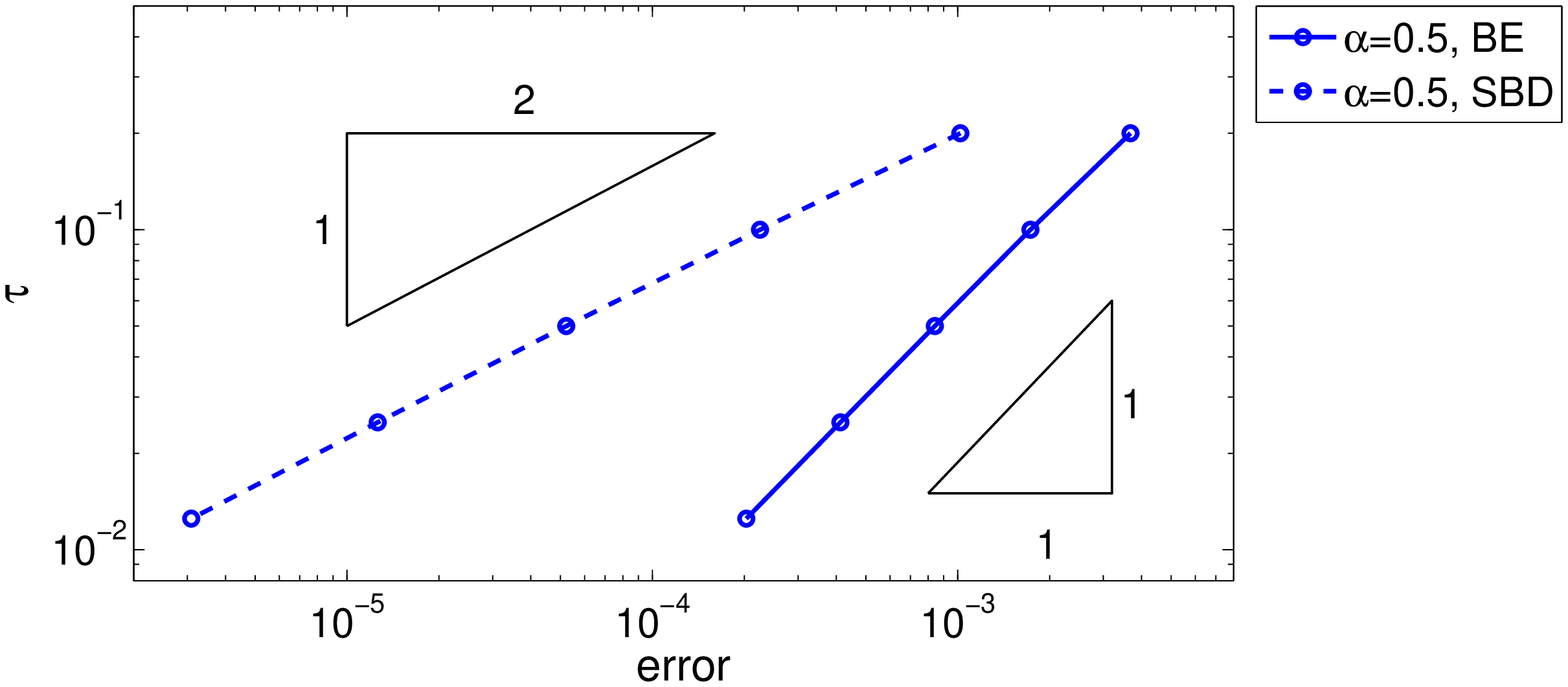}
  \caption{Error plots for example (a) at $t=0.1$, with $\alpha=0.5$ and $h=2^{-11}$.}
  \label{fig:smooth_time}
\end{figure}

In Table \ref{tab:smooth_space} and Figure \ref{fig:smooth_space}, we show the $L^2\II$- and $H^1\II$-norms
of the error at $t=0.1$ for the BE scheme. We set $\tau=2\times10^{-5}$
and check the spatial convergence rate. The numerical results show $O(h^{2})$ and $O(h)$ convergence
rates respectively for the $L^2\II$- and $H^1\II$-norms of the error, which fully confirm Theorem
\ref{thm:smooth-initial-op}. Further, the empirical convergence rate is almost independent of the
fractional order $\alpha$.
\begin{table}[h!]
\caption{Error for example (a): $t=0.1$,  $h = 2^{-k}$
and $\tau=5\times10^{-5}$. }\label{tab:smooth_space}
\begin{center}
\begin{tabular}{@{}|c|c|ccccc|c|@{}}
     \hline
      $\al$ & $k$ & $3$ & $4$ & $5$ &$6$ & $7$ & rate\\
     \hline
     $\al=0.1$ & $L^2$-norm & 6.16e-4 &1.59e-4  &4.00e-5  &9.90e-6  &2.38e-6  & $\approx 2.01$ ($2.00$) \\
     & $H^1$-norm           & 1.19e-2 & 5.99e-3 & 2.99e-3 & 1.49e-3 & 7.26e-4 & $\approx 1.01$ ($1.00$)  \\
     \hline
     $\al=0.5$ & $L^2$-norm  & 1.58e-3 &4.00e-4  &1.00e-4  &2.48e-5  &5.95e-6 & $\approx 2.01$ ($2.00$)  \\
     & $H^1$-norm          & 3.92e-2 & 1.98e-2 & 9.88e-3 & 4.91e-3 & 2.40e-3  & $\approx 1.01$ ($1.00$)  \\
     \hline
     {$\al=0.9$} & $L^2$-norm   & 1.38e-3 &3.47e-4  &8.67e-5  &2.15e-5  &5.16e-6 & $\approx 2.01$ ($2.00$)\\
     & $H^1$-norm           & 3.56e-2 & 1.79e-2 & 8.96e-3 & 4.45e-3 & 2.17e-3  & $\approx 1.01$ ($1.00$)\\
     \hline
     \end{tabular}
\end{center}
\end{table}

\begin{figure}[htb!]
  \includegraphics[trim = 1cm .1cm 2cm 0.0cm, clip=true,width=10cm]{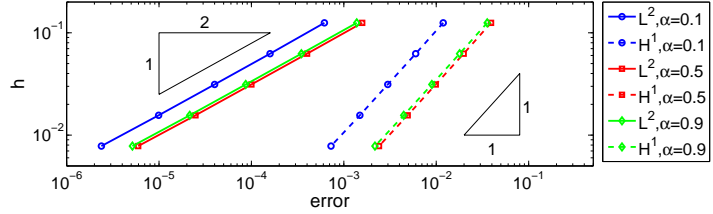}
  \caption{Error for example (a): $t=0.1$, $\tau = 2\time10^{-5}$,
  $\alpha=0.1$, $0.5$ and $0.9$.}
  \label{fig:smooth_space}
\end{figure}

\subsection{Numerical results for example (b)}
In Tables \ref{tab:nonsmooth_time} and \ref{tab:nonsmooth_space} we present the results for example (b). The 
temporal convergence rate is $O(\tau)$ and $O(\tau^2)$ for the BE and the SBD method, respectively, cf.
Table \ref{tab:nonsmooth_time}, and the spatial convergence rate is of order $O(h^2)$ in $L^2\II$-norm and 
$O(h)$ in $H^1\II$-norm, cf. Table \ref{tab:nonsmooth_space}. For nonsmooth initial data, we are especially 
interested in errors for $t$ close to zero. Thus we also present the error at $t = 0.01$ and $t = 0.001$ in 
Table \ref{tab:nonsmooth_space}. The numerical results fully confirm the predicted rates.

Further, in Table \ref{tab:check_singular} and Figure \ref{fig:singular} we show the $L^2\II$-norm of the 
error for examples (a) and (b), for fixed $h = 2^{-6}$ and $t\to 0$. To check the spatial discretization error, 
we fix time step $\tau$ at $\tau=t/1000$ and use the SBD method so that
the temporal discretization error is negligible. We observe that in the smooth case, i.e., example (a), the spatial
error essentially stays unchanged, whereas in the nonsmooth case, i.e., example (b), it deteriorates as $t\rightarrow0$.
In example (b) the initial data $v\in \dot H^{1/2-\epsilon}\II$ for any
$\epsilon>0$, and by Remark \ref{rem:int_fully2}, the error grows like $O(t^{-3\al/4})$ as $t\rightarrow0$.
The empirical rate in Table \ref{tab:check_singular} and Figure \ref{fig:singular} agrees well with the
theoretical prediction, i.e., $-3\al/4=-0.375$ for $\al=0.5$.

\begin{table}[htb!]
\caption{The $L^2\II$-norm of the error for example (b)
 at $t=0.1$, with $h=2^{-11}$.}
\label{tab:nonsmooth_time}
\begin{center}
\vspace{-.3cm}
     \begin{tabular}{|c|c|ccccc|c|}
     \hline
     & $\tau$ & $1/5$ & $1/10$ &$1/20$ &$1/40$ & $1/80$  &rate \\
     \hline
     BE
     & $\al=0.1$      &2.82e-2 &1.42e-2 &7.13e-3 &3.56e-3 &1.76e-3   &$\approx$ 1.00 (1.00) \\
     \cline{2-7}
     & $\al=0.5$       &8.67e-3 &4.18e-3 &2.05e-3 &1.01e-3 &4.97e-4  &$\approx$ 1.02 (1.00) \\
     \cline{2-7}
     & $\al=0.9$      &9.06e-4 &4.47e-4 &2.21e-4 &1.09e-4 &5.42e-5   &$\approx$ 1.02 (1.00) \\
     \hline
     SBD
     & $\al=0.1$      &7.14e-3 &1.61e-3 &3.92e-4 &9.63e-5 &2.38e-5   &$\approx$ 2.05 (2.00)\\
     \cline{2-7}
     & $\al=0.5$      &2.46e-3&5.05e-4 &1.17e-4 &2.82e-5 &6.91e-6  &$\approx$ 2.06 (2.00) \\
     \cline{2-7}
     & $\al=0.9$      &1.67e-4 &3.58e-5 &8.40e-6 &2.04e-6 &5.11e-7  &$\approx$ 2.08 (2.00) \\
     \hline
     \end{tabular}
\end{center}
\end{table}

\begin{table}[h!]
\caption{Error for example (b): $\al=0.5$, 
$h = 2^{-k}$ and $N=1000$.
}\label{tab:nonsmooth_space}
\begin{center}
\begin{tabular}{@{}|c|c|ccccc|c|@{}}
     \hline
      $t$ & $k$ & $3$ & $4$ & $5$ &$6$ & $7$ & rate\\
     \hline
     $t=0.1$ & $L^2$-norm & 1.63e-3 &4.09e-4  &1.02e-4  &2.55e-5  &6.30e-6  & $\approx 2.00$ ($2.00$) \\
     & $H^1$-norm           & 4.04e-2 & 2.02e-2 & 1.01e-2 & 5.04e-3 & 2.51e-3 & $\approx 1.00$ ($1.00$)  \\
     \hline
     {$t=0.01$} & $L^2$-norm  & 5.87e-3 &1.47e-3  &3.66e-4  &9.13e-5  &2.26e-5 & $\approx 2.00$ ($2.00$)  \\
     & $H^1$-norm          & 1.62e-1 & 8.08e-2 & 4.04e-2 & 2.02e-2 & 1.00e-2  & $\approx 1.00$ ($1.00$)  \\
     \hline
     {$t=0.001$} & $L^2$-norm   & 1.47e-2 & 3.66e-3 &9.15e-4  &2.28e-4  &5.65e-5 & $\approx 2.00$ ($2.00$)\\
     & $H^1$-norm           & 4.48e-1 & 2.24e-1 & 1.12e-1 & 5.60e-2 & 2.78e-2  & $\approx 1.00$ ($1.00$)\\
     \hline
     \end{tabular}
\end{center}
\end{table}

\begin{table}[htb!]
\caption{The $L^2\II$-norm of the error for examples (a) and (b) with $\al=0.5$, $h=2^{-6}$,
and $t \to 0$.}\label{tab:check_singular}
\center
     \begin{tabular}{|c|cccccc|c|}
     \hline
     $t$&   1e-3 & 1e-4 & 1e-5 & 1e-6 & 1e-7 & 1e-8  &rate\\
     \hline
     (a)  & 2.48e-4 & 3.07e-4 & 3.27e-4 & 3.46e-4 & 3.55e-4 & 3.58e-4 & $\approx$ -0.02 (0) \\
     (b)  & 2.28e-4 & 5.07e-4 & 1.22e-3 & 2.89e-3 & 6.78e-3 & 1.56e-2 & $\approx$ -0.37 (-0.37) \\
     \hline
     \end{tabular}
\end{table}

\begin{figure}[htb!]
  \includegraphics[trim = 1cm .1cm 2cm 0.0cm, clip=true,width=8cm]{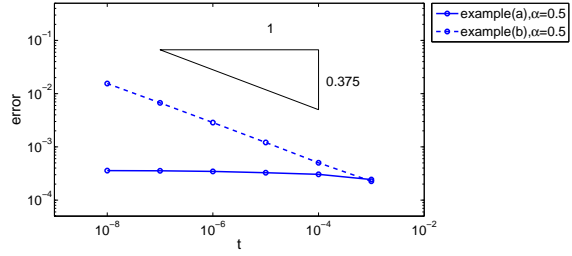}
  \caption{Error plots for examples (a) and (b) with $h=2^{-6}$,
  $\alpha=0.5$ for $t\rightarrow0$.}\label{fig:singular}
\end{figure}

\subsection{Numerical results for example (c)}
In the case of very weak data, according to Remarks \ref{rem:weak1} and \ref{rem:weak2},
we can only expect spatial convergence for a small time step size $\tau$. The results in Table \ref{tab:weak_space}
indicate a superconvergence phenomenon with a rate $O(h^2)$ in the $L^2\II$-norm and $O(h)$ in the $H^1\II$-norm.
This is attributed to the fact that in one dimension the solution with the
Dirac $\delta$-function as the initial data is smooth from both sides of the support point and
the finite element spaces $X_h$ have good approximation property. When the singularity
point $x=1/2$ is not aligned with the grid, Table \ref{tab:weak_space2} shows an $O(h^{3/2})$
and $O(h^{1/2})$ rate for the $L^2\II$- and $H^1\II$-norm of the error, respectively.

\begin{table}[h!]
\caption{Error  for  example (c):  $\al=0.5$,
$h = 2^{-k}$, and $N=1000$. }\label{tab:weak_space}
\begin{center}
\begin{tabular}{@{}|c|c|ccccc|c|@{}}
     \hline
      $t$ & $k$ & $3$ & $4$ & $5$ &$6$ & $7$ & rate\\
     \hline
     $t=0.1$ & $L^2$-norm & 1.19e-4 &2.98e-5  &7.45e-6  &1.86e-6  &4.62e-7  & $\approx 2.00$ ($1.50$) \\
     & $H^1$-norm           & 5.35e-3 & 2.69e-3 & 1.35e-3 & 6.72e-4 & 3.34e-4 & $\approx 1.00$ ($0.50$)  \\
     \hline
     {$t=0.01$} & $L^2$-norm  & 2.41e-3 &6.04e-4  &1.51e-4  &3.77e-5  &9.31e-6 & $\approx 2.00$ ($1.50$)  \\
     & $H^1$-norm          & 3.98e-2 & 1.99e-2 & 9.92e-3 & 4.95e-3 & 2.46e-3  & $\approx 1.00$ ($0.50$)  \\
     \hline
     {$t=0.001$} & $L^2$-norm   & 1.25e-2 & 3.12e-3 &7.80e-4  &1.94e-4  &4.83e-5 & $\approx 2.00$ ($1.50$)\\
     & $H^1$-norm           & 5.00e-1 & 2.50e-1 & 1.25e-1 & 6.23e-2 & 3.09e-2  & $\approx 1.00$ ($0.50$)\\
     \hline
     \end{tabular}
\end{center}
\end{table}

\begin{table}[h!]
\caption{Error for example (c): $\al=0.5$,
$h = 1/(2^{k}+1)$ and $N=1000$.}\label{tab:weak_space2}
\begin{center}
\begin{tabular}{@{}|c|c|ccccc|c|@{}}
     \hline
      $t$ & $k$ & $3$ & $4$ & $5$ &$6$ & $7$ & rate\\
     \hline
     $t=0.1$ & $L^2$-norm & 5.84e-3 &2.22e-3  &8.15e-4  &2.93e-4  &1.04e-4  & $\approx 1.50$ ($1.50$) \\
     & $H^1$-norm           & 1.79e-1 & 1.29e-1 & 9.16e-2 & 6.44e-2 & 4.45e-2 & $\approx 0.52$ ($0.50$)  \\
     \hline
     {$t=0.01$} & $L^2$-norm  & 2.42e-2 &9.54e-3  &3.57e-3  &1.30e-3  &4.63e-4 & $\approx 1.48$ ($1.50$)  \\
     & $H^1$-norm          & 7.77e-1 & 5.68e-1 & 4.07e-1 & 2.87e-1 & 1.98e-1  & $\approx 0.51$ ($0.50$)  \\
     \hline
     {$t=0.001$} & $L^2$-norm   & 8.01e-2 & 3.27e-2 &1.25e-2  &4.57e-3  &1.64e-3 & $\approx 1.46$ ($1.50$)\\
     & $H^1$-norm           & 2.65e0 & 1.97e0 & 1.43e0 & 1.02e0 & 7.05e-1  & $\approx 0.49$ ($0.50$)\\
     \hline
     \end{tabular}
\end{center}
\end{table}

\subsection{Numerical results for example (d)}
Here we consider a two-dimensional example on the unit square $\Omega =(0,1)^2$ for the nonsmooth
initial data. To discretize the problem, we divide the unit interval $(0,1)$ into $K=2^k$
equally spaced subintervals with a mesh size $h=1/K$ so that the domain is divided
into $K^2$ small squares. We get a symmetric triangulation of the domain by connecting
the diagonal of each small square. Table \ref{tab:nonsmooth_time2D} shows a temporal convergence
rate of first order and second order for the BE and SBD method, respectively.
Spatial errors at $t=0.1$, $0.01$ and $0.001$ are showed in Table \ref{tab:nonsmooth_space2D},
which imply a convergence with a rate of $O(h^2)$ in the $L^2\II$-norm and $O(h)$ in the $H^1\II$-norm.
In Figure \ref{fig:2D_time} and \ref{fig:2D_space} we plot the results shown in Tables
\ref{tab:nonsmooth_time2D} and \ref{tab:nonsmooth_space2D}, respectively.
All numerical results confirm our convergence theory.

\begin{table}[htb!]
\caption{The $L^2$-norm of the error for example (d) at $t=0.1$, with $\al=0.5$ and
$h=2^{-9}$.}
\label{tab:nonsmooth_time2D}
\begin{center}
\vspace{-.3cm}
     \begin{tabular}{|c|c|ccccc|c|}
     \hline
     & $\tau$ & $1/5$ & $1/10$ &$1/20$ &$1/40$ & $1/80$  &rate \\
     \hline
     BE
     & $\al=0.5$       &4.53e-3 &2.15e-3 &1.04e-3 &5.17e-4 &2.56e-4  &$\approx$ 1.03 (1.00) \\
     \hline
     SBD
     & $\al=0.5$      &1.33e-3 &2.80e-4 &6.48e-5 &1.56e-5 &3.79e-6  &$\approx$ 2.11 (2.00) \\
     \hline
     \end{tabular}
\end{center}
\end{table}

\begin{figure}[htb!]
  \includegraphics[trim = 1cm .1cm 2cm 0.0cm, clip=true,width=8cm]{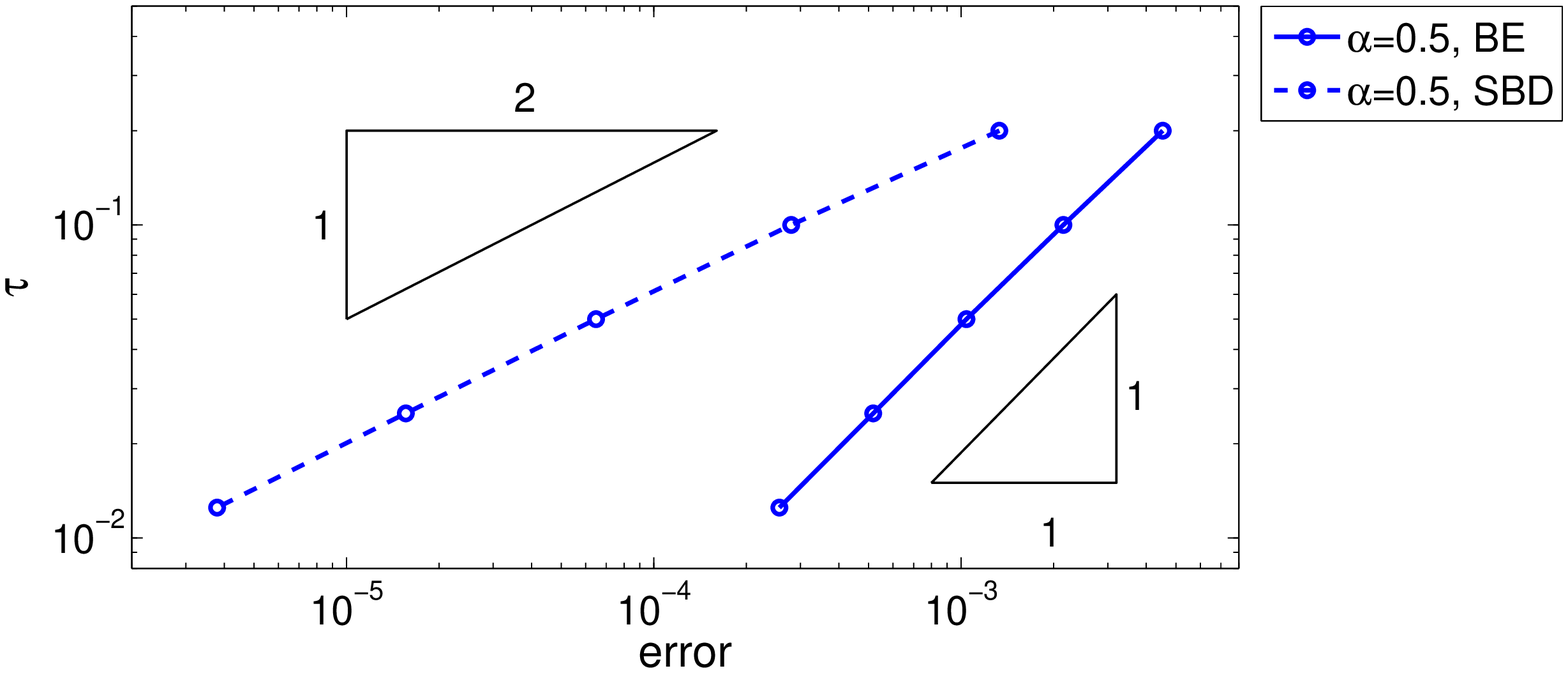}
  \caption{Error plots for example (d) at $t=0.1$ with $\alpha=0.5$ and $h=2^{-9}$.}
  \label{fig:2D_time}
\end{figure}

\begin{table}[h!]
\caption{Error for example (d): $\al=0.5$, $h = 2^{-k}$ and $N=1000$.}\label{tab:nonsmooth_space2D}
\begin{center}
\begin{tabular}{@{}|c|c|ccccc|c|@{}}
     \hline
      $t$ & $k$ & $3$ & $4$ & $5$ &$6$ & $7$ & rate\\
     \hline
     $t=0.1$ & $L^2$-norm & 1.95e-3 &5.02e-4  &1.26e-4  &3.12e-5  &7.61e-6  & $\approx 2.01$ ($2.00$) \\
     & $H^1$-norm           & 3.29e-2 & 1.63e-2 & 8.11e-3 & 4.03e-3 & 1.97e-3 & $\approx 1.00$ ($1.00$)  \\
     \hline
     $t=0.01$ & $L^2$-norm  & 7.79e-3 &2.00e-3  &5.03e-4  &1.25e-4  &2.98e-5 & $\approx 2.02$ ($2.00$)  \\
     & $H^1$-norm          & 1.43e-1 & 7.09e-2 & 3.53e-2 & 1.75e-2 & 8.56e-3  & $\approx 1.01$ ($1.00$)  \\
     \hline
     $t=0.001$ & $L^2$-norm   & 1.97e-2 & 5.09e-3 &1.28e-3  &3.19e-4  &7.05e-5 & $\approx 2.00$ ($2.00$)\\
     & $H^1$-norm           & 4.44e-1 & 2.22e-1 & 1.11e-1 & 5.52e-2 & 2.69e-2  & $\approx 1.01$ ($1.00$)\\
     \hline
     \end{tabular}
\end{center}
\end{table}

\begin{figure}[htb!]
  \includegraphics[trim = 1cm .1cm 2cm 0.0cm, clip=true,width=10cm]{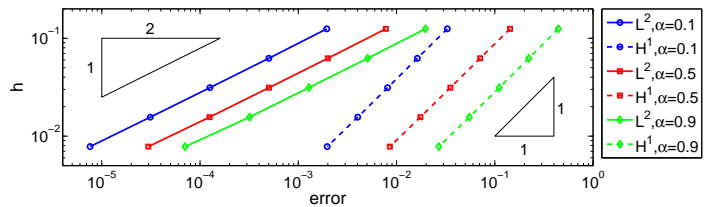}
  \caption{Error plots of example (d):
  $\alpha=0.1, 0.5, 0.9$ and $N=1000$ at $t=0.1$, $0.01$ and $0.001$.}\label{fig:2D_space}
\end{figure}

\section{Concluding remarks}
In this work, we have studied the homogeneous problem for the Rayleigh-Stokes equation in a second grade
generalized flow. The Sobolev regularity of the solution was established using an operator theoretic
approach. A space semidiscrete scheme based on the Galerkin finite element method and two fully discrete
schemes based on the backward Euler method and second-order backward difference method and related
convolution quadrature were developed and optimal with respect to the data regularity error estimates
were provided for both semidiscrete and fully discrete schemes. Extensive numerical experiments
fully confirm the sharpness of our convergence analysis.

\section*{Acknowledgements}
The authors are grateful to Prof. Christian Lubich for his helpful comments on an earlier version
of the paper, which led to a significant improvement of the presentation in Section \ref{sec:fullydiscrete},
and an anonymous referee for many constructive comments.
The research of B. Jin has been supported by NSF Grant DMS-1319052, and R. Lazarov was supported in parts
by NSF Grant DMS-1016525 and also by Award No. KUS-C1-016-04, made by King Abdullah University of Science and Technology (KAUST).

\bibliographystyle{abbrv}
\bibliography{frac}

\end{document}